\documentclass{amsart}

\title{Day algebras}
\author{Edmund Robinson and Joshua Wrigley}
\thanks{The first author would like to acknowledge the support of the UK EPSRC research grant EP/R006865/1.  The second author acknowledges the support of the Agence Nationale de la Recherche (ANR), project ANR-23-CE48-0012-01.
}
\date{}
\dedicatory{This paper is dedicated to Pino Rosolini, bringer of joy.}

\address[Edmund Robinson]{School of Electronic Engineering and Computer Science, Queen Mary University of London,  London,  E1 4FZ, UK.}
\email{e.p.robinson@qmul.ac.uk}
\urladdr{https://www.eecs.qmul.ac.uk/~edmundr/}

\address[Joshua Wrigley]{IRIF, Université Paris Cité, Paris, 75013, France.}
\email{wrigley@irif.fr}
\urladdr{https://jlwrigley.github.io/}

\usepackage{amssymb}
\usepackage{amsmath}
\usepackage{amsthm}
\usepackage{mathtools}
\usepackage{mathbbol}
\usepackage{thmtools}

\usepackage[pdfencoding=auto,pdfusetitle]{hyperref}
\hypersetup{colorlinks=true, linkcolor={black}, citecolor={black}, urlcolor={black}}

\usepackage[noabbrev,capitalise]{cleveref}
\crefformat{equation}{(#2#1#3)}
\crefformat{enumi}{#2#1#3}
\crefformat{enumii}{#2#1#3}

\usepackage{tikz-cd}
\usepackage{tikz}
\usetikzlibrary{matrix,positioning}
\usetikzlibrary{calc}
\usetikzlibrary{arrows}
\tikzset{> =stealth}

\usepackage{quiver}

\theoremstyle{plain}
\newtheorem{thm}{Theorem}[section]
\newtheorem{lem}[thm]{Lemma}
\newtheorem{coro}[thm]{Corollary}
\newtheorem{prop}[thm]{Proposition}

\theoremstyle{definition}
\newtheorem{df}[thm]{Definition}

\newtheorem{rem}[thm]{Remark}

\newtheorem{ex}[thm]{Example}

\renewcommand{\phi}{\varphi}
\newcommand{\N}{\mathbb{N}}

\newcommand{\Ccat}{\mathbb{C}}
\newcommand{\Dcat}{\mathbb{D}}
\newcommand{\Ecat}{\mathbb{E}}

\renewcommand{\H}{\mathbb{H}}
\newcommand{\W}{{W}}
\renewcommand{\a}{{\bf a}}

\newcommand{\Set}{{\bf Set}}
\newcommand{\Prof}{{\bf Prof}}
\newcommand{\Cat}{{\bf Cat}}
\newcommand{\1}{\mathbb{1}}
\newcommand{\End}{{\bf End}}
\renewcommand{\partial}{{\rm par}}
\newcommand{\Span}{{\bf Span}}

\newcommand{\pb}{{\rm pb}}

\newcommand{\op}{^{\rm op}}
\newcommand{\id}{1}
\newcommand{\2}{\mathbb{2}}

\newcommand{\slashedrightarrow}{\, {\mathrel{\ooalign{\hss$\vcenter{\hbox{\tikz{\draw (0,4.5pt) -- (0,0) ;}}}\mkern2.75mu$\hss\cr$\to$}}}\, }

\newcommand\sepimp{\mathrel{-\mkern-6mu*}}

\begin{document}
	
	\begin{abstract}
		In this paper we show that the Day monoidal product generalises in a straightforward way to other algebraic constructions and partial algebraic constructions on categories. This generalisation was motivated by its applications in logic, for example in hybrid and separation logic. We use the description of the Day monoidal product using profunctors to show that the definition generalises to an extension of an arbitrary algebraic structure on a category to a pseudo-algebraic structure on a functor category. We provide two further extensions. First we consider the case where some of the operations on the category are partial, and second we show that the resulting operations on the functor category have adjoints (they are residuated). 
	\end{abstract}

	\maketitle

	
	\section*{Introduction}

At its heart, this is a paper that could have been written some time ago.  In 1970, Brian Day \cite{day} showed that a monoidal structure on a (small) category $\Ccat$ produced a corresponding monoidal structure on the functor category $[\Ccat \to \Set]$.  This construction is now known as the Day monoidal product, or Day convolution. 

Although Day's construction is often expressed using technology familiar only to category theorists, when we view it from an abstract perspective its existence is not really a surprise. The topos $[\Ccat \to \Set]$ can be viewed as a completion by colimits of $\Ccat\op$, and so it is not a surprise that structure on $\Ccat\op$ extends to structure on the completion. But it is not completely straightforward. Given functors $F,G: \Ccat \to \Set$, we want to define $F\otimes_D G: \Ccat \to \Set$. The data we have is 
\[\begin{tikzcd}
				{\Ccat \times \Ccat} & \Ccat \\
				{\Set \times \Set} & \Set
				\arrow["\otimes", from=1-1, to=1-2]
				\arrow["{\langle F, G \rangle}"', from=1-1, to=2-1]
				\arrow["\times", from=2-1, to=2-2]
			\end{tikzcd}\]
These functors do not compose, but there is a canonical choice of functor $\Ccat\to\Set$ on the right, given by left Kan extension, and this is how the Day monoidal product is defined. The left Kan extension can also be presented as a coend, which is how the Day monoidal product is often itself presented. 

The purpose of this paper is to point out that Day's construction generalises to other algebra structures on $\Ccat$, and to discuss how it generalises when the operations in those structures are partial. 

Others have also pointed out that if we translate this diagram into the language of profunctors, an idea already present in Day's work of 1973 \cite{day2006closed}, then since the bicategory of profunctors is self-dual, we can reverse the direction of the $\otimes$, and take the composite. This gives a third equivalent definition of Day convolution.  But then we have a construction that involves operations on $\Ccat$ followed by a sequence of functors in parallel, followed by products in $\Set$, and that makes it obvious that Day's construction generalises to arbitrary algebra structures on $\Ccat$.

This is a significant fact in the context of categorical logic. Most approaches to logic deal with relations $M\vDash P$, expressing the validity of a logical proposition $P$ in a context $M$. Such relations always introduce a partial order on contexts: $M \leqslant M'$ if and only if for all (possibly for all atomic) $P$, if $M\vDash P$ then $M'\vDash P$, and then the interpretation of (atomic) propositions is functorial from the partially ordered set of contexts to the partial order $\2 = (\bot \leqslant \top)$. One particularly important instance of this is the construction of the Kripke models, used to give semantics to intuitionistic and other forms of logic. In a Kripke model, the context $M$ is described as a possible world $w$ and the relation $w\vDash P$ is written $w\Vdash P$. The set of worlds $W$ carries a partial order, and, for logics that are intuitionistic rather than classical, interpretations are expected to satisfy the monotonicity property: if $w\leqslant w'$ and $w\Vdash P$, then $w'\Vdash P$. 
As a result, the logic of the Kripke model lives naturally in the internal logic of the topos $[W\to\Set]$. The interpretation of propositions corresponds to subobjects of the terminal object, and the interpretation of the propositional connectives is as in the topos of functors, e.g.\ $w \Vdash P \rightarrow Q$ if and only if, for each $w \leqslant v$, $v \Vdash P$ implies that $v \Vdash Q$ (see e.g.\ \cite{modallogic} or \cite{goldblatt2014topoi}, and for modal operations, i.e.\ $\square$ and $\lozenge$, see \cite{kavvos}).

Viewed from this perspective, the question we are addressing is: ``how does an algebra structure on the category of worlds translate into the interpretation of a novel logical connective that can be used to reason about propositions with Kripke semantics?''

This is not an abstract question. One of the most important forms of computational logic developed in recent years is \emph{separation logic} \cite{ishtiaq-ohearn,Ohearn:ACM-separation-logic}. Separation logic is designed to reason about programs running in the standard execution model for modern languages that employs pointers and heap structures. Simplifying slightly, the semantics for separation logic is given by a Kripke model in which the worlds are \emph{heaps}. In the original version they form a discrete category, but in intuitionistic variants they are ordered by inclusion. We reason about a relation $h\Vdash \phi$, where $h$ is a heap, and $\phi$ is a property of it. The key innovation of separation logic is the separating conjunction $\ast$, which is interpreted as the Day monoidal product for a binary operation $\sqcup$ representing the union of \emph{disjoint} heaps. In the context of propositional logic, the coend formula simplifies, and we have that
\[h\Vdash \phi\ast\psi \iff \exists h_1,h_2.\ h_1 \sqcup h_2 \leqslant h\mbox{ and }h_1\Vdash\phi\mbox { and }h_2\Vdash\psi.\]
However, $\sqcup$ is a partial operation. It is only defined when the heaps involved are disjoint. This is to ensure that a properties true of two sub-heaps can be combined to make a property true of their union without requiring a side condition of consistency on the intersection. That in turn feeds into the ability to reason structurally about large systems. Given a small piece of code, properties can be determined with respect to the part of the heap that it actually touches, and then extended to properties valid for the whole codebase. This is one of the key technical properties that enables the verification of code at scale in the context of an entire codebase.  This definition itself lives in the context of a logic, \emph{bunched implication logic} (BI), which has a standard interpretation in Kripke semantics over a resource monoid \cite{PYM2004257} where there are two notions of combining resources, sharing, corresponding to conventional logic, and non-sharing, corresponding to a substructural logic. 

This is not the only logic interpreted in a Kripke structure where the clauses follow a similar pattern.  \emph{Hybrid logic} includes a syntax for naming worlds \cite{blackburn2014henkin}, for example a unary operation $@_w$ on propositions where 
\[w' \Vdash @_{w}\phi \iff w\Vdash \phi\]
Hybrid and separation operations can be seen combined in recent work on reasoning about advanced algorithms in extensions of concurrent separation logic (see e.g.\ \cite{Brotherston:reasoning-over-permissions}).

\subsection*{Overview}
The paper is laid out as follows. 
\begin{enumerate}
    \item In \cref{sec:preliminaries}, we recall the basic technologies we will be using, namely profunctors and operads. We choose operads to formalise multi-variable equational theories as they enable relatively clear accounts of both the essential treatment of substitution and the coherence required when moving from a purely equational to a pseudo-algebraic theory. 

    \item In \cref{sec:day_algebras} we give the basic treatment for extension of algebra structures. The key result is \cref{thm:day_ext_is_operadic}, which establishes a morphism of operads: 		
\[(-)_D \colon \End(\Ccat) \to \End([\Ccat \to \Set])\op,\]
from which the extension of any given algebra structure on $\Ccat$ follows as an immediate corollary. 

    \item \cref{sec:partial} contains two extensions, one to partial operations and the other to adjoints of them, establishing that these operations are ``residuated'', they support forms of implication, just like the Day monoidal product.  We also discuss in more detail how the logical connectives of hybrid, linear temporal and separation logic are recovered.  As explained further in \cref{remarks_on_spans}, we must be precise about how we interpret partial algebra structure, as our construction only preserves \emph{strong} equality of partial operations (see \cref{df:stongequality}).
\end{enumerate}
We conclude with an Appendix giving details of the coherence required.

	\section{Preliminaries}\label{sec:preliminaries}
	
	In this paper, a category $\Ccat$ is assumed to be \emph{small}.
	If $\Ccat$ and $\Dcat$ are a pair of categories, $\Ccat \times \Dcat$ will always refer to their product \emph{as categories}, i.e.\ the product in the (2-)category of categories $\Cat$.  This will be important when we consider other (bi-)categories whose objects are categories but whose arrows are not functors.  Similarly, $\Ccat^n$ will refer to the $n$-fold product of $\Ccat$ with itself and $\1$ for the one object, one arrow category, i.e.\ the empty product.

    We will also compromise between categorical and type-theoretic traditions by using $[\Ccat\to\Dcat]$ (in square brackets) to denote the category of functors from $\Ccat$ to $\Dcat$. 
	
	\subsection{Profunctors}
	Our generalisation of the Day monoidal product will be constructed using \emph{profunctors}, which we recall here.
	\begin{df}
		Let $\Ccat$ and $\Dcat$ be categories.  A \emph{profunctor} $F \colon \Ccat \slashedrightarrow \Dcat$ is a functor $F \colon \Dcat\op \times \Ccat \to \Set$.
	We will use $\Ccat \to \Dcat$ to denote a functor and $\Ccat \slashedrightarrow \Dcat$ to denote a profunctor.
    
	   Profunctors generalise relations. If the categories $\Ccat$ and $\Dcat$ are discrete, and then a profunctor that takes values in the subterminals $\2 \subseteq \Set$ corresponds to a relation between the object sets. Profunctors can be composed using a \emph{coend} \cite[\S IX.6]{CWM}, which in the special case above is precisely the relational composite.  
		
        Given a pair of profunctors $F \colon \Ccat \slashedrightarrow \Dcat$, $G \colon \Dcat \slashedrightarrow \Ecat$, their composite $G \circ F \colon \Ccat \slashedrightarrow \Ecat$ acts on objects $(e,c) \in \Ecat\op \times \Ccat$ by the coend
		\[
		G \circ F(e,c) = \int^{d \in \Dcat} G(e,d) \times F(d,c).
		\]
		Concretely, this is the quotient set
		\[
		\left(\coprod_{d \in \Dcat} G(e,d) \times F(d,c) \right){/{\sim}}
		\]
		where $\sim$ is the equivalence relation generated by $(g,F(\alpha,\id_c)(f)) \sim (G(\id_c,\alpha)(g),f)$, for all $f \in F(d,c)$, $g \in G(e,d')$ and $\alpha \colon d \to d' \in \Dcat$.  On arrows $(\beta,\gamma) \colon (e,c) \to (e',c')$, $G \circ F \colon \Ccat \slashedrightarrow \Ecat$ acts by the universally induced arrow, which concretely sends the equivalence class of the pair $(g,f) \in G(e,d) \times F(d,c)$ to the equivalence class of the pair $(G(\beta,\id_d)(g),G(\id_d,\gamma)(f))$.
		
		Since coends, and hence the composite $G \circ F$, are only defined up to natural isomorphism, this definition naturally gives rise not to a 2-category but a \emph{bicategory} of profunctors $\Prof$.  A 2-cell $F \Rightarrow G$ between profunctors is just a natural transformation of the underlying functors.  Bicategories generalise 2-categories in that the associative and unit rules are only satisfied up to (coherent) isomorphism (see \cite[\S 1]{benaboubicats}).
	\end{df}
	There are two ways to consider a functor $F \colon \Ccat \to \Dcat$ as a profunctor:
	\begin{enumerate}
		\item $\Dcat(\id_\Dcat,F) \colon \Ccat \slashedrightarrow \Dcat$, the profunctor that sends $(d,c) \in \Dcat\op \times \Ccat$ to $\Dcat(d,F(c))$; and, analogously, 
		\item $\Dcat(F,\id_\Dcat) \colon \Dcat \slashedrightarrow \Ccat$, the profunctor that sends $(c,d) \in \Ccat\op\times\Dcat$ to $\Dcat(Fc,d)$.
	\end{enumerate}
	These definitions extend to bifunctors $\Cat \to \Prof$ and $\Cat\op \to \Prof$.  Importantly, this means that we can apply functors `in reverse' by considering functors as profunctors.
	\begin{ex}\label{ex:composite_of_cat_and_catop}
		Let $F \colon \Ccat \to \Dcat, G \colon \Ecat \to \Dcat$ be functors.  It is an easy exercise to show that the composite
		\[
		\begin{tikzcd}[column sep = large]
			\Ccat \ar["\shortmid"{marking}]{r}{\Dcat(\id_\Dcat,F)} & \Dcat \ar["\shortmid"{marking}]{r}{\Dcat(G, \id_\Dcat)} & \Ecat
		\end{tikzcd}
		\]
		is isomorphic to the profunctor $\Dcat(G,F) \colon \Ecat\op \times \Ccat \to \Set$ that sends $(e,c) \in \Ecat\op \times \Ccat$ to $\Dcat(G(e),F(c))$.
	\end{ex}
	
	Moreover, a functor $P \colon \Ccat \to \Set$, i.e.\ a \emph{co-presheaf}, can be regarded as a functor
	\[
	P \colon \1\op \times \Ccat \to \Set
	\]
	and hence as a profunctor $P \colon \Ccat \slashedrightarrow \1$.  Indeed, the category of co-presheaves, denoted by $[\Ccat \to \Set]$, is isomorphic to the category $\Prof(\Ccat,\1)$.

	\subsection{Operads}  Our applications to universal algebra will be framed in the language of \emph{operad theory}.  An operad formalises the composition of $n$-ary operations.  For the most part, we follow the conventions of \cite[\S 2]{leinsteroperads}.  
	\begin{df}\label{df:operad}
		A \emph{operad} $P$ consists of the following data.
		\begin{enumerate}
			\item For each $n \in \N$, a set $P(n)$, whose elements we call \emph{$n$-ary operations}.
			\item A distinguished element $1 \in P(1)$ called the \emph{unit}.
			\item For each $n,k_1,\dots,k_n \in \N$, a function
			\[
			\circ \colon P(n) \times P(k_1) \times \dots \times P(k_n) \to P(k_1 + \dots + k_n)
			\]
			which we call \emph{multi-composition} (our reasons for this departure in terminology from \cite{leinsteroperads} will be explained in 
            \cref{df:cat-operad}).
		\end{enumerate}
		These must satisfy the \emph{unital} condition that
		\[ \theta \circ (1, \dots ,1) = \theta = 1 \circ \theta ,\]
		for all $\theta \in P(n)$, and the \emph{associativity} condition
		\[
		\theta \circ (\theta_1 \circ (\theta_1^1, \dots , \theta_1^{k_1}), \dots , \theta_n \circ (\theta_n^1 , \dots , \theta_n^{k_n})) = (\theta \circ (\theta_1, \dots , \theta_n)) \circ (\theta_1^1, \dots , \theta_n^{k_n}),
		\]
		for all $\theta, \theta_i , \theta_i^j$ such that these multi-composites exist.
	\end{df}
	\begin{ex}
		Below, we define the \emph{algebras} for an operad and briefly sketch how operad theory is used in universal algebra.  To that end, we describe here the `syntactic operads' that are used to describe single-sorted algebraic theories.  Let $\Sigma$ be a collection of operation symbols with arities, called a \emph{signature}.  We use $P_{\Sigma}$ to denote the operad freely generated by the operations in $\Sigma$: for each $n \in \N$, $P_\Sigma(n)$ is the set of terms over $\Sigma$ with $n$ inputs.  For example, suppose $\Sigma$ consists of a unary operation $\theta$ and a binary operation $\psi$, then $P_\Sigma(1)$ contains $1$, $\theta$, $\theta^2$, $\psi \circ (\theta,\theta)$, etc.  If $E$ is a collection of equations $t_1 = t_2$ between terms over $\Sigma$, i.e.\ an \emph{algebraic theory}, we can quotient each $P_{\Sigma}(n)$ by the congruence generated by $E$ to obtain an operad $P_{(\Sigma,E)}$, the \emph{classifying operad} for the pair $(\Sigma,E)$.  Thus, $P_{(\Sigma,E)}(n)$ is the set of terms over $\Sigma$, with $n$ free variables, up to provable equivalence modulo $E$.
	\end{ex}
	Since we will encounter constructions that are defined only up to (canonical) isomorphism, we require a notion of operad that includes `transformations' between operations.  For our purposes, it suffices to consider operads \emph{enriched} over $\Cat$.  Essentially, we replace the words `set' and `function' in \cref{df:operad} with, respectively, `category' and `functor'.
	\begin{rem}
		Whenever working in higher dimensions, a choice must be made about whether composition/identities/etc.\ are defined up to equality or isomorphism.  We do not work in the greatest generality, instead picking and choosing to suit our needs: our `operads with transformations' will be strict, but their morphisms will be `pseudo'/`lax', in the sense that the unit and multi-composites are preserved up to (coherent) isomorphism/transformation.  In contrast, \emph{pseudo-operads}, where the unital and associativity conditions hold only up to isomorphism, are discussed in \cite{dayross-pseudooperads}.
	\end{rem}
	\begin{df}\label{df:cat-operad}
		A \emph{$\Cat$-operad} $P$ is a $\Cat$-enriched operad, and so consists of:
		\begin{enumerate}
			\item a category $P(n)$, for each $n \in \N$,
			\item a \emph{unit} $1 \in P(1)$,
			\item and, for each $n,k_1,\dots,k_n \in \N$, a functor
			\[
			\circ \colon P(n) \times P(k_1) \times \dots \times P(k_n) \to P(k_1 + \dots + k_n).
			\]
			Naming this functor \emph{multi-composition} distinguishes it from the composition of arrows in $P(n)$.
		\end{enumerate}
		These must satisfy the analogous unital and associativity conditions from \cref{df:operad}.
	\end{df}
	\begin{ex}[Endomorphism operad]\label{ex:operand_End}
		Our archetypal example is the $\Cat$-operad of \emph{$n$-ary endomorphisms} on a category $\Ccat$, which we denote by $\End(\Ccat)$.  For each $n \in \N$, $\End(\Ccat)(n)$ is the category $[\Ccat^n \to \Ccat]$ of functors $\Ccat^n \to \Ccat$ and natural transformations between these.  The unit is the identity functor $\id_\Ccat$.  Multi-composition is given by composition of $n$-ary functors, as suggested by our choice of terminology.
	\end{ex}
	\begin{df}
		If $P$ is a $\Cat$-operad, we use $P\op$ to denote the $\Cat$-operad where $P\op(n) = P(n)\op$.  The unit and composition remain the same (modulo taking the opposite category).
	\end{df}
	Operads, and $\Cat$-operads, can be used in universal algebra via the notion of an \emph{algebra} for an operad.
	\begin{df}\label{df:pseudomorph_of_cat-operads}
		A \emph{lax-morphism} of $\Cat$-operads $\omega \colon P \to Q$ consists of a functor $\omega_n \colon P(n) \to Q(n)$ for each $n$ such that the unit and multi-composition are preserved up to coherent transformation, by which we mean that:
		\begin{enumerate}
			\item there is a morphism $\lambda \colon 1 \to \omega_1 (1) \in Q(1)$,
			\item and, for each $\theta \in P(n), \theta_1 \in P(k_1), \dots , \theta_n \in P(k_n)$, a morphism
			\[
			\eta_{\theta, \theta_1, \dots , \theta_n} \colon \omega_n(\theta) \circ (\omega_{k_1}(\theta_1), \dots , \omega_{k_n}(\theta_n)) \to \omega_{k_1 + \dots + k_n} (\theta \circ (\theta_1 , \dots , \theta_n)) 
			\]
			in $Q(k_1 + \dots + k_n)$, natural in $\theta, \theta_i$,
		\end{enumerate}
		and moreover satisfying coherence axioms.  While necessary, the coherence axioms are notationally cumbersome and somewhat unenlightening, and so they are postponed to \cref{appendix:coherence}.  We will call such an $\omega \colon P \to Q$ a \emph{pseudo-morphism} if $\lambda$ and each $\eta_{\theta, \theta_1, \dots , \theta_n}$ are isomorphisms, and a \emph{strict morphism} if they are identities.
	\end{df}
	\begin{df}\label{df:pseudo-algebra}
		An \emph{algebra} for a $\Cat$-operad $P$ consists of a category $\Ccat$ and a (strict) morphism of $\Cat$-operads $P \to \End(\Ccat)$.  We define \emph{pseudo-algebras} and \emph{lax-algebras} similarly.
	\end{df}
	In the case of a syntactic operad, this returns the notion of algebra we would have expected.  Let $\Sigma$ be a collection of $n$-ary operations and $E$ a collection of equations over $\Sigma$.  An algebra of the syntactic operad $P_{(\Sigma,E)}$ consists of a category $\Ccat$ with a choice of $n$-ary operation $\Ccat^n \to \Ccat$ for each operation symbol in $\Sigma$ such that the equations in $E$ are satisfied.  A pseudo-algebra is instead the case where these equations hold only up to (coherent) isomorphism.
	\begin{ex}[Monoidal categories]\label{ex:monoidal_cats}
		Monoidal categories can be presented either by giving an $n$-ary tensor $\otimes_n$ for each $n$, or by giving just the binary ($\otimes$) and nullary (unit) operations. Consider the signature $\Sigma_M$ containing an $n$-ary operation $\otimes_n$ for each $n$ and the collection $E_M$ of equations $\otimes_1 = 1$ and $\otimes_n \circ (\otimes_{k_1}, \dots , \otimes_{k_n}) = \otimes_{k_1 + \dots + k_2}$.  An algebra $P_{\Sigma_M, E_M} \to \Ccat$ is a \emph{strict monoidal category}, a pseudo-algebra is a \emph{monoidal category}, and a lax-algebra a \emph{lax monoidal category} \cite[\S 1]{dayross-pseudooperads}. (The usual definition of a (strict/{skew}) monoidal category, such as in \cite{skewmonoidal}, can be recovered by taking $\otimes_0$ as the unit and $\otimes_2$ as the binary product, see \cite[\S 7]{bourkelack}.)  
	\end{ex}

	\section{Day algebras}\label{sec:day_algebras}
	In this section, we present the first of the two main constructions of the paper: a pseudo-morphism of operads $\End(\Ccat) \to \End([\Ccat \to \Set])\op$ from the endomorphism operad on a category $\Ccat$ to the opposite endomorphism operad on the category of \emph{co-presheaves} $[\Ccat \to \Set]$.  Thus, if $\Ccat$ has the structure of an pseudo-algebra for an operad $P$, i.e.\ if there is a pseudo-morphism of $\Cat$-operads $P \to \End(\Ccat)$, the structure can be \emph{extended} to $[\Ccat \to \Set]$ by taking the composite pseudo-morphism $P \to \End(\Ccat) \to \End([\Ccat \to \Set])\op$.
	\subsection{Day monoidal product} 
	Before presenting the main construction of this paper, we recall the construction of the \emph{Day convolution} from \cite{day}.  Given monoidal structure on a category $\Ccat$, there is a canonical extension of the monoidal structure to the category of co-presheaves $[\Ccat \to \Set]$.
	\begin{df}[Day monoidal product]\label{df:day_product}
		Let $(\Ccat,\otimes,I)$ be a monoidal category.  Given two functors $F, G \colon \Ccat \to \Set$, their \emph{Day product} $F \otimes_D G \colon \Ccat \to \Set$ can be defined in three equivalent ways.
		\begin{enumerate}
			\item It is the \emph{left Kan extension} of the functor $F \times G \colon \Ccat \times \Ccat \to \Set$ along $\otimes \colon \Ccat \times \Ccat \to \Ccat$:
			\[\begin{tikzcd}
				{\Ccat \times \Ccat} & \Ccat \\
				{\Set \times \Set} & \Set.
				\arrow["\otimes", from=1-1, to=1-2]
				\arrow["{\langle F, G \rangle}"', from=1-1, to=2-1]
				\arrow["{F \otimes_D G}", dashed, from=1-2, to=2-2]
				\arrow[shorten <=10pt, shorten >=10pt, Rightarrow, from=2-1, to=1-2]
				\arrow["\times", from=2-1, to=2-2]
			\end{tikzcd}\]
			\item The left Kan extension can be computed as the coend
			\[
			(F \otimes_D G)(c) = \int^{(c_0,c_1) \in \Ccat \times \Ccat} F(c_0) \times G(c_1) \times \Ccat(c_0 \otimes c_1,c).
			\]
			\item Finally, and most importantly for us, this is precisely the composite of profunctors
			\[
			\begin{tikzcd}
				\Ccat \ar["\shortmid"{marking}]{r}{\Ccat(\otimes,\id_{\Ccat})} & \Ccat \times \Ccat \ar["\shortmid"{marking}]{r}{F \times G} &  \1.
			\end{tikzcd}
			\]
		\end{enumerate}
		Similarly, we define $I_D$ as the co-representable $\Ccat(I,\id_\Ccat) \colon \Ccat \to \Set$, i.e.\ the profunctor $\Ccat(I,\id_\Ccat) \colon \Ccat \slashedrightarrow \1$.
	\end{df}
	\begin{prop}[Theorem 3.3 \cite{day}]
		If $(\Ccat,\otimes,I)$ is a monoidal category, then so too is $([\Ccat \to \Set],\otimes_D,I_D)$.
	\end{prop}
	
	\subsection{Day extension}
	We make the rather obvious observation that there is nothing special about the monoidal product used in defining its Day extension in \cref{df:day_product}.  Any algebraic structure on $\Ccat$ can be extended to one on the category of co-presheaves $[\Ccat \to \Set]$, if we allow equations to hold up to isomorphism.  To describe the Day extension of an algebraic structure, we have two tasks:
	\begin{enumerate}
		\item first, we construct the extension of arbitrary $n$-ary operations,
		\item second, we must show that equations between operations are preserved.
	\end{enumerate}
	
	\begin{df}[Day extension]\label{df:day-extension}
		Let $\theta \colon \Ccat^n \to \Ccat$ be a functor.  The \emph{Day extension} of $\theta$ is the functor $\theta_D \colon [\Ccat \to \Set]^n \to [\Ccat \to \Set]$ that acts on objects by sending $(F_1, \dots, F_n)$ to the functor $\Ccat \to \Set$ corresponding to the composite of profunctors
		\[
		\begin{tikzcd}[column sep = large]
			\Ccat \ar["\shortmid"{marking}]{r}{\Ccat(\theta,1_\Ccat)} & \Ccat^n \ar["\shortmid"{marking}]{r}{F_1 \times \dots \times F_n} & \1.
		\end{tikzcd}
		\]
		On arrows, $\theta_D$ sends a tuple of natural transformations $(\alpha_1 \colon F_1 \Rightarrow G_1, \dots , \alpha_n \colon F_n \Rightarrow G_n)$ to the natural transformation corresponding to the 2-cell of profunctors
		\[\begin{tikzcd}[column sep=large]
			\Ccat & {\Ccat^n} && {\1}.
			\arrow["{\Ccat(\theta,1_\Ccat)}", "\shortmid"{marking}, from=1-1, to=1-2]
			\arrow[""{name=0, anchor=center, inner sep=0}, "{F_1 \times \dots \times F_n}", "\shortmid"{marking}, curve={height=-18pt}, from=1-2, to=1-4]
			\arrow[""{name=1, anchor=center, inner sep=0}, "{G_1 \times \dots \times G_n}"', "\shortmid"{marking}, curve={height=18pt}, from=1-2, to=1-4]
			\arrow["{\prod_{i=1}^n \alpha_i}", shift right=3, shorten <=5pt, shorten >=5pt, Rightarrow, from=0, to=1]
		\end{tikzcd}\]
		\end{df}
	
	\begin{rem}
		Note that the Day extension $\theta_D(F_1, \dots , F_n)$ described above is equivalent to:
		\begin{enumerate}
			\item the coend
			\[
			\theta_D(F_1, \dots , F_n)(c) = \int^{(c_1, \dots , c_n) \in \Ccat^n} F_1(c_1) \times \dots \times F_n(c_n) \times \Ccat(\theta(c_1, \dots , c_n),c) ,
			\]
			\item and the left Kan extension
			\[\begin{tikzcd}
				\Ccat^n & \Ccat \\
				\Set^n & \Set,
				\arrow["\theta", from=1-1, to=1-2]
				\arrow["{\langle F_1 , \dots , F_n \rangle}"', from=1-1, to=2-1]
				\arrow["{\theta_D(F_1, \dots, F_n)}", dashed, from=1-2, to=2-2]
				\arrow[shorten <=10pt, shorten >=10pt, Rightarrow, from=2-1, to=1-2]
				\arrow["\prod", from=2-1, to=2-2]
			\end{tikzcd}
			\]
		\end{enumerate}
		mirroring \cref{df:day_product}.
	\end{rem}
	
	Our Day extension is naturally contravariant for $n$-ary operations.
		Let $\alpha \colon \theta \Rightarrow \chi$ be a natural transformation between two functors $\theta, \chi \colon \Ccat^n \to \Ccat$.  We define $\alpha_D \colon \chi_D \Rightarrow \theta_D$ as the natural transformation whose component at $(F_1, \dots , F_n) \in [\Ccat \to \Set]^n$ is given by the map of co-presheaves corresponding to the transformation of profunctors
		\[\begin{tikzcd}[column sep=large]
			{\Ccat} && \Ccat^n & { \1}.
			\arrow[""{name=0, anchor=center, inner sep=0}, "{\Ccat(\chi,1_\Ccat)}", "\shortmid"{marking}, curve={height=-18pt}, from=1-1, to=1-3]
			\arrow[""{name=1, anchor=center, inner sep=0}, "{\Ccat(\theta,1_\Ccat)}"', "\shortmid"{marking}, curve={height=18pt}, from=1-1, to=1-3]
			\arrow["{F_1 \times \dots \times F_n}", "\shortmid"{marking}, from=1-3, to=1-4]
			\arrow["{\Ccat(\alpha,1_\Ccat)}", shift right=3, shorten <=5pt, shorten >=5pt, Rightarrow, from=0, to=1]
		\end{tikzcd}\]
		Naturality follows from the interchange law for horizontal composition of 2-cells.
	
	\begin{lem}\label{lem:Day_ext_func_fixed_n}
		For a fixed $n$, the Day extension 
		\[(-)_D \colon [\Ccat^n \to \Ccat] \to [[\Ccat \to \Set]^n \to [\Ccat \to \Set]]\op\]
		defines a functor.
	\end{lem}
	\begin{proof}
		This follows from the fact that the bifunctor $\Cat\op \to \Prof$ that sends $F \colon \Ccat \to \Dcat$ to $\Dcat(F,1_\Dcat) \colon \Dcat \slashedrightarrow \Ccat$ is functorial in 2-cells.
	\end{proof}

	\subsection{Day extension as a morphism of operads}
	Next, we demonstrate that equations between operations are preserved by Day extension (up to canonical isomorphism). The language of {operads} gives a neat way of formulating this.

	\begin{thm}\label{thm:day_ext_is_operadic}
		Day extension defines a pseudo-morphism of $\Cat$-operads 
		\[(-)_D \colon \End(\Ccat) \to \End([\Ccat \to \Set])\op,\]
		where each $n$-ary functor $\theta \colon \Ccat^n \to \Ccat$ is sent to $\theta_D \colon [\Ccat \to \Set]^n \to [\Ccat \to \Set]$.
	\end{thm}

    Many readers will be accustomed to viewing Day extension in the context of presheaves, as in the following corollary: 

    \begin{coro}\label{coro:day_ext_is_operadic_presheaves}
        Day extension defines a pseudo-morphism of $\Cat$-operads 
		\[(-)_D \colon \End(\Dcat) \to \End([\Dcat\op \to \Set]),\]
		where each $n$-ary functor $\theta \colon \Dcat^n \to \Dcat$ is sent to $\theta_D \colon [\Dcat \to \Set]^n \to [\Dcat \to \Set]$.
    \end{coro}

    \begin{proof}[Proof of Corollary \ref{coro:day_ext_is_operadic_presheaves}]
     If $\Ccat = \Dcat\op$, then $\End(\Ccat) = \End(\Dcat)\op$.
    \end{proof}

	\begin{proof}[Proof of Theorem \ref{thm:day_ext_is_operadic}]
		By \cref{lem:Day_ext_func_fixed_n}, $(-)_D$ defines a functor
		\[[\Ccat^n \to \Ccat] \to [[\Ccat \to \Set]^n \to [\Ccat \to \Set]]\op\]
		for each $n$.  It remains to show that the unit and multi-composition are preserved up to coherent isomorphism.  To show that the unit is preserved, note that $\Ccat(1_\Ccat, 1_\Ccat) \colon \Ccat \slashedrightarrow \Ccat$ is the identity profunctor on $\Ccat$. By unwrapping definitions, $({\id_{\Ccat}})_D $ acts on $[\Ccat \to \Set]$ as $- \circ \Ccat(1_\Ccat,1_\Ccat)$.  Thus, the \emph{right unitor} $\mathfrak{r}_{\Ccat, \1}$ of the bicategory $\Prof$ (see \cite[Definition 1.1]{benaboubicats}) yields the desired coherent isomorphism $- \circ \Ccat(1_\Ccat,1_\Ccat)  = ({\id_{\Ccat}})_D \to \id_{[\Ccat \to \Set]} \in \End([\Ccat \to \Set])$ (and hence an isomorphism $\id_{[\Ccat \to \Set]} \to ({\id_{\Ccat}})_D$ in $\End([\Ccat \to \Set])\op$).  
		
		Now we turn to multi-composition.  Given $\theta \colon \Ccat^n \to \Ccat$ and $\theta_i \colon \Ccat^{k_i} \to \Ccat$, we must show that there is a canonical isomorphism $(\theta \circ (\theta_1, \dots, \theta_n))_D \cong \theta_D \circ ({\theta_1}_D, \dots {\theta_n}_D)$.  Given co-presheaves $F_1^1, \dots , F_1^{k_1}, F_2^1,  \dots , F_n^{k_n}$, consider the composable triple of profunctors
        \begin{align}
			& \begin{tikzcd}[column sep = large, ampersand replacement = \&]
				\Ccat \ar["\shortmid"{marking}]{r}{\Ccat(\theta,1_\Ccat)} \& \Ccat^n \ar["\shortmid"{marking}]{rr}{\Ccat^n((\theta_1,\dots,\theta_n),1_{\Ccat^n})} \&\& \Ccat^{k_1 + \dots + k_n} \ar["\shortmid"{marking}]{r}{F_1^1 \times \dots \times F_n^{k_n}} \& \1
			\end{tikzcd} \nonumber \\
        = & \begin{tikzcd}[column sep = large, ampersand replacement = \&]
			\Ccat \ar["\shortmid"{marking}]{r}{\Ccat(\theta,1_\Ccat)} \& \Ccat^n \ar["\shortmid"{marking}]{rr}{\Ccat(\theta_1,1_\Ccat) \times \dots \times \Ccat(\theta_n, 1_\Ccat)} \&\& \Ccat^{k_1 + \dots + k_n} \ar["\shortmid"{marking}]{r}{F_1^1 \times \dots \times F_n^{k_n}} \& \1,
            \label{eq:theta_circ_theta}
		\end{tikzcd}
        \end{align}
		where the equality follows since hom-sets in a product of categories are products of hom-sets, i.e.\ $\Ccat^n((\theta_1,\dots,\theta_n),1_{\Ccat^n}) = \Ccat(\theta_1,1_\Ccat) \times \dots \times \Ccat(\theta_n, 1_\Ccat)$.  There are two ways to compose this triple -- either left pair first, or right pair first.  The former yields a canonical isomorphism
        \begin{align*}
            & (F_1^1  \times \dots \times F_n^{k_n}) \circ( \Ccat^n((\theta_1, \dots , \theta_n),1_{\Ccat^n}) \circ \Ccat(\theta,1_\Ccat) ) \\
               \cong &  (F_1^1  \times \dots \times F_n^{k_n}) \circ \Ccat(\theta \circ (\theta_1, \dots , \theta_n),1_{\Ccat}) \\
            = & (\theta \circ (\theta_1, \dots, \theta_n))_D(F_1^1, \dots , F_n^{k_n})
        \end{align*}
        given by $(F_1^1 \times \dots \times F_n^{k_n}) \ast \mathfrak{c}_{\theta, (\theta_1, \dots , \theta_n)}$, where \[\mathfrak{c}_{\theta, (\theta_1, \dots , \theta_n)} \colon \Ccat^n((\theta_1, \dots , \theta_n),1_{\Ccat^n}) \circ \Ccat(\theta,1_\Ccat) \xRightarrow{\sim} \Ccat(\theta \circ (\theta_1, \dots , \theta_n), 1_\Ccat)\]
        is the coherent choice of isomorphism given by the bifunctoriality of $\Cat\op \to \Prof$. Next, recall that ${\theta_i}_D(F_i^1, \dots , F_i^{k_i})$ is the co-presheaf corresponding to the profunctor $(F_i^1 \times \dots \times F_i^{k_i}) \circ \Ccat(\theta_i,1_\Ccat)$, whence we deduce that, by composing the right pair first in \cref{eq:theta_circ_theta}, we obtain the composable pair of profunctors
		\[
		\begin{tikzcd}[column sep = large]
			\Ccat \ar["\shortmid"{marking}]{r}{\Ccat(\theta,1_\Ccat)} & \Ccat^n \ar["\shortmid"{marking}]{rrrr}{{\theta_1}_D(F_1^1, \dots, F_1^{k_1}) \times \dots \times {\theta_n}_D(F_n^1, \dots , F_n^{k_n})} &&&& \1 ,
		\end{tikzcd}
		\]
        i.e.\ the co-presheaf $ \theta_D \circ ({\theta_1}_D(F_1^1, \dots, F_1^{k_1}) , \dots , {\theta_n}_D(F_n^1, \dots , F_n^{k_n}))$.  We conclude that there is a coherent isomorphism 
         \begin{align*}
            & \theta_D \circ ({\theta_1}_D(F_1^1, \dots, F_1^{k_1}) , \dots , {\theta_n}_D(F_n^1, \dots , F_n^{k_n})) \\
            = &  ((F_1^1  \times \dots \times F_n^{k_n}) \circ \Ccat^n((\theta_1, \dots , \theta_n),1_{\Ccat^n}) )\circ \Ccat(\theta,1_\Ccat) \\
            \cong & (F_1^1  \times \dots \times F_n^{k_n}) \circ( \Ccat^n((\theta_1, \dots , \theta_n),1_{\Ccat^n}) \circ \Ccat(\theta,1_\Ccat) )
        \end{align*}
		given by the \emph{associator} $\mathfrak{a}_{\Ccat,\Ccat^n,\1} $ of the bicategory $\Prof$.  Hence, the composite isomorphism $((F_1^1 \times \dots \times F_n^{k_n}) \ast \mathfrak{c}_{\theta, (\theta_1, \dots , \theta_n)} )
        \circ \mathfrak{a}_{\Ccat,\Ccat^n,\1} $ yields the coherent isomorphism
        \[(\theta \circ (\theta_1, \dots, \theta_n))_D(F_1^1, \dots , F_n^{k_n}) \xrightarrow{\sim} \theta_D \circ ({\theta_1}_D(F_1^1, \dots, F_1^{k_1}) , \dots , {\theta_n}_D(F_n^1, \dots , F_n^{k_n})).\]
        It is easily shown that this isomorphism is natural in both $F_i^j$ and $\theta, \theta_i$.  By taking the inverse of the above natural isomorphism, we obtain the isomorphism in the desired direction $\theta_D \circ ({\theta_1}_D , \dots , {\theta_n}_D) \to (\theta \circ (\theta_1, \dots , \theta_n))_D \in \End([\Ccat \to \Set])\op$.  
       
        Finally, the coherence axioms expressed in \cref{appendix:coherence} follow as a direct consequence of the coherence equations satisfied by the bicategory $\Prof$ and the bifunctoriality of $\Cat\op \to \Prof$ (see \cref{prop:coherence_of_ext}).
	\end{proof} 
	
	In particular, $(-)_D \colon \End(\Ccat) \to \End([\Ccat \to \Set])\op$ sends invertible 2-cells to invertible 2-cells. It follows that Day extension send algebras in $\Ccat$ to pseudo-algebras in $[\Ccat\to\Set]$ (and similarly $[\Ccat\op\to\Set]$). 

    \begin{coro}\label{coro:Day-algebras}
       Let $\Sigma$ be an algebraic signature (i.e.\ a collection of operations with finite arity) and $E$ a collection of equations $E$ over $\Sigma$. Given a pseudo-algebra $\Ccat$ of $(\Sigma,E)$, then the Day extension of each operation $\theta \colon \Ccat^n \to \Ccat$ in $\Sigma$ defines the structure of a pseudo-algebra for $(\Sigma,E)$ on $[\Ccat \to \Set]$. 
    \end{coro}
   Indeed, the algebra structure is just the composite pseudo-morphism $P_{\Sigma,E} \to \End(\Ccat) \to \End([\Ccat \to \Set])\op$.

    \begin{rem}
        After writing this paper, Nathanael Arkor pointed out to the authors that the statement of \cref{coro:Day-algebras} appears in \cite[Proposition 1.4.28]{espalungue}.
    \end{rem}
	
	\begin{ex}[\S 7 \cite{dayross-pseudooperads}]
		In a similar fashion, postcomposing with $(-)_D$ sends lax algebra structure to \emph{op-lax} algebra structure.  In particular, given a lax monoidal category, i.e.\ a lax-morphism of operads $P_{\Sigma_M, E_M} \to \End(\Ccat)$ (cf.\ \cref{ex:monoidal_cats}), its Day convolution yields an \emph{op-lax} morphism of operads $P_{\Sigma_M, E_M} \to \End([\Ccat\to\Set])$, and thus an op-lax monoidal structure on $[\Ccat \to \Set]$.
	\end{ex}

	\section{Partial operations}\label{sec:partial}
	There are important instances where the algebra operations are partial rather than total.  For example, in \emph{separation logic}, the key monoidal operation represents the union of \emph{disjoint} subheaps \cite{separation}.  In this section, we show that the Day extension considered in \cref{sec:day_algebras} can be modified to incorporate partial operations as well.	
	\subsection{The partial endomorphism operad}
	First, we will describe an operadic approach to algebraic structures with partial operations.  When working with partial operations, we must make a choice about what it means to compose partial functors, i.e.\ spans of functors $\Dcat \hookleftarrow \Ecat \to \Ccat$, which we write as $\Dcat \rightharpoondown \Ccat$.  We adopt the convention that the composite of a pair of partial functors exists if there is a factorisation
	\begin{equation}\label{eq:composition_of_partial_functors}
	\begin{tikzcd}
		& \Ecat \ar{r}{F} \ar[hook]{d} & \Ccat \\
		\Ecat' \ar{r}[']{G} \ar[hook]{d} \ar[dashed]{ru}{G'} & \Dcat & \\
		\Dcat',
	\end{tikzcd}
	\end{equation}
	i.e.\ $G$ only takes values in the domain of $F$, in which case the composite is given by the outside span $\Dcat \hookleftarrow \Ecat' \to \Ccat$ (for further discussion on this choice, see \cref{remarks_on_spans}).  As such, we must relax the notion of a $\Cat$-operad to allow for situations where the multi-composite may not exist.
	\begin{df}(Partial $\Cat$-operads and their morphisms)
		A \emph{partial $\Cat$-operad} consists of:
		\begin{enumerate}
			\item a category $P(n)$, for each $n \in \N$,
			\item a \emph{unit} $1 \in P(1)$,
			\item for each $n,k_1,\dots,k_n \in \N$, a \emph{partial} functor
			\[
			P(n) \times P(k_1) \times \dots \times P(k_n) \rightharpoondown P(k_1 + \dots + k_n)
			\]
		\end{enumerate}
		satisfying the associativity axioms
		\[
		\theta \circ (\theta_1 \circ (\theta_1^1, \dots , \theta_1^{k_1}), \dots , \theta_n \circ (\theta_n^1 , \dots , \theta_n^{k_n})) = (\theta \circ (\theta_1, \dots , \theta_n)) \circ (\theta_1^1, \dots , \theta_n^{k_n}),
		\]
		\[
		\alpha \circ (\alpha_1 \circ (\alpha_1^1, \dots , \alpha_1^{k_1}), \dots , \alpha_n \circ (\alpha_n^1 , \dots , \alpha_n^{k_n})) = (\alpha \circ (\alpha_1, \dots , \alpha_n)) \circ (\alpha_1^1, \dots , \alpha_n^{k_n}),
		\]
		for all objects $\theta, \theta_i , \theta_i^j$ and arrows $\alpha, \alpha_i , \alpha_i^j$ such that these multi-composites exist, and the partial unital axiom: not only is $ 1 \circ \theta$ is defined for all $\theta$, but moreover it is given by $\theta$.
		
		We define morphisms of partial $\Cat$-operads in a similar way to morphisms of $\Cat$-operads (\cref{df:pseudomorph_of_cat-operads}), except that we only preserve those multi-composites that are defined.  That is to say, a \emph{lax morphism of partial $\Cat$-operads} $\omega \colon P \to Q$ consists of the data of a functor $\omega_n \colon P(n) \to Q(n)$ for each $n$, a morphism $\lambda \colon 1 \to \omega_1(1) \in Q(1)$, and for each $\theta, \theta_i$ such that the multi-composite $\theta \circ (\theta_1 , \dots , \theta_n)$ exists, then the multi-composite $\omega_n(\theta) \circ (\omega_{k_1}(\theta_1), \dots , \omega_{k_n}(\theta_n))$ is also defined and there is a morphism
		\[
		\eta_{\theta, \theta_1, \dots , \theta_n} \colon \omega_n(\theta) \circ (\omega_{k_1}(\theta_1), \dots , \omega_{k_n}(\theta_n)) \to \omega_{k_1 + \dots + k_n} (\theta \circ (\theta_1 , \dots , \theta_n)) 
		\]
		in $Q(k_1 + \dots + k_n)$.  This data must satisfy the naturality and coherence conditions expressed in \cref{df:pseudomorph_of_cat-operads} and \cref{appendix:coherence}.  Once again, we make the distinction between lax, pseudo and strict morphisms as in \cref{df:pseudomorph_of_cat-operads}.
	\end{df}
    Initially, we restrict to partial operations whose domain is a full subcategory.
	\begin{df}(Partial operation)
		Let $\Ccat$ be a category.  An $n$-ary \emph{partial operation} on $\Ccat$, written as $\theta \colon \Ccat^n \rightharpoondown \Ccat$, is a functor $\theta \colon \Dcat \to \Ccat$ where $\Dcat$ is a \emph{full} subcategory of $\Ccat^n$.
	\end{df}
	\begin{df}[Partial endomorphism operad]\label{df:partial_endo_operad}
		Let $\Ccat$ be a category.  We denote by $\End(\Ccat)_\partial$ the partial $\Cat$-operad of \emph{partial operations}.  For each $n \in \N$, $\End(\Ccat)_\partial(n)$ is the category whose objects are $n$-ary partial operations $\theta \colon \Ccat^n \rightharpoondown \Ccat$.  An arrow $\alpha \colon \theta \to \chi$ in $\End(\Ccat)_\partial$ is a natural transformation $\alpha \colon \theta \Rightarrow \chi$ between partial operations with the same domain.
		
		The unit of the operad is the identity functor $\id_\Ccat$.  For $n, k_1, \dots k_n \in \N$, the multi-composite
		\[
		\circ \colon \End(\Ccat)_\partial(n) \times \End(\Ccat)_\partial(k_1) \times \dots \times \End(\Ccat)_\partial(k_n) \rightharpoondown \End(\Ccat)_\partial(k_1 + \dots + k_n)
		\]
		is given as follows.  The domain is the full subcategory spanned by those partial operations $\theta \colon \Ccat^n \rightharpoondown \Ccat, \theta_i \colon \Ccat^{k_i} \rightharpoondown \Ccat$ for which $(\theta_1, \dots , \theta_n) \colon \Ccat^{k_1 + \dots k_n} \rightharpoondown \Ccat^n$ factors through the domain $\Dcat \subseteq \Ccat^n$ of $\theta$, i.e.\ those tuples where the composite as partial functors exists.  The composite as partial functors describes the action on objects of the multi-composition $\circ$. Let $\theta , \chi \colon \Ccat \rightharpoondown \Ccat$, $\theta_i, \chi_i \colon \Ccat^{k_i} \rightharpoondown \Ccat$ be partial operations for which the multi-composites $\theta \circ (\theta_1 , \dots , \theta_n), \chi \circ (\chi_1, \dots , \chi_n)$ exist, and let $\alpha \colon \theta \to \chi , \beta_i \colon \theta_i \to \chi_i $ be arrows in $\End(\Ccat)_\partial(n)$ (respectively, $\End(\Ccat)_\partial(k_i)$).  That is, the pairs $\theta,\chi$ and $\theta_i,\chi_i$ all share the same domain. Note that since $(\theta_1, \dots, \theta_n), (\chi_1, \dots , \chi_n) \colon \Dcat_1 \times \dots \times \Dcat_n \to \Ccat$ both factor through the full subcategory $\Dcat \subseteq \Ccat$ (the shared domain of $\theta$ and $\chi$), the natural transformation $(\beta_1, \dots , \beta_n) \colon (\theta_1, \dots , \theta_n) \Rightarrow (\chi_1, \dots , \chi_n)$ lifts to a natural transformation between the factoring functors $(\theta_1, \dots , \theta_n)' , (\chi_1, \dots , \chi_n)' \colon \Dcat_1 \times \dots \times \Dcat_n \to \Dcat$.  We define $\alpha \circ (\beta_1 ,\dots , \beta_n)$ as the horizontal composite of 2-cells
        \[\begin{tikzcd}
	{\Ccat^{k_1 + \dots + k_n}} & {\Dcat_1 \times \dots \times \Dcat_n} &[20pt]& \Dcat && \Ccat.
	\arrow[hook', from=1-2, to=1-1]
	\arrow[""{name=0, anchor=center, inner sep=0}, "{(\theta_1, \dots , \theta_n)'}", curve={height=-25pt}, from=1-2, to=1-4]
	\arrow[""{name=1, anchor=center, inner sep=0}, "{(\chi_1, \dots , \chi_n)'}"', curve={height=25pt}, from=1-2, to=1-4]
	\arrow[""{name=2, anchor=center, inner sep=0}, "\theta", curve={height=-18pt}, from=1-4, to=1-6]
	\arrow[""{name=3, anchor=center, inner sep=0}, "\chi"', curve={height=18pt}, from=1-4, to=1-6]
	\arrow["{(\beta_1,\dots,\beta_n)}", shorten <=5pt, shorten >=5pt, Rightarrow, from=0, to=1]
	\arrow["\alpha", shorten <=5pt, shorten >=5pt, Rightarrow, from=2, to=3]
\end{tikzcd}\]
	\end{df}
	\begin{ex}
		Each total operation $\theta \colon \Ccat^n \to \Ccat$ can be considered as a partial operation $\theta \colon \Ccat^n \rightharpoondown \Ccat$ where the inclusion of the domain is just the identity $\id_{\Ccat^n} \colon \Ccat^n \to \Ccat^n$.  Indeed, this identifies $\End(\Ccat)(n)$ as a full subcategory of $\End(\Ccat)_\partial(n)$, for each $n \in \N$.  These inclusion functors together define a (strict) morphism of partial $\Cat$-operads $\End(\Ccat) \hookrightarrow \End(\Ccat)_\partial$, and in this sense $\End(\Ccat)$ is a \emph{suboperad} of $\End(\Ccat)_\partial$.
	\end{ex}
	\subsection{Day extension of partial operations}
	\begin{df}[Day extension of partial operations]\label{df:partial-Day-extension}
		The \emph{Day extension} of a partial operation $\theta \colon \Ccat^n \rightharpoondown \Ccat$ is the \emph{total} functor $\theta_D \colon [\Ccat \to \Set]^n \to [\Ccat \to \Set]$ given by sending $F_1, \dots , F_n \colon \Ccat \to \Set$ to the co-presheaf corresponding to the composite of profunctors
		\[
		\begin{tikzcd}[column sep = large]
			\Ccat \ar["\shortmid"{marking}]{r}{\Ccat(\theta,1_\Ccat)} & \Dcat \ar["\shortmid"{marking}]{r}{\Ccat^n(1_{\Ccat^n},i)}  & \Ccat^n \ar["\shortmid"{marking}]{r}{F_1 \times \dots \times F_n} & \1,
		\end{tikzcd}
		\]
		and a tuple of natural transformations $\alpha_i \colon F_i \Rightarrow G_i$ between co-presheaves to the natural transformation corresponding to the transformation of profunctors
		\[\begin{tikzcd}[column sep=large]
			\Ccat & \Dcat \ar["\shortmid"{marking}]{r}{\Ccat^n(1_{\Ccat^n},i)} & {\Ccat^n} && {\1},
			\arrow["{\Ccat(\theta,1_\Ccat)}", "\shortmid"{marking}, from=1-1, to=1-2]
			\arrow[""{name=0, anchor=center, inner sep=0}, "{F_1 \times \dots \times F_n}", "\shortmid"{marking}, curve={height=-18pt}, from=1-3, to=1-5]
			\arrow[""{name=1, anchor=center, inner sep=0}, "{G_1 \times \dots \times G_n}"', "\shortmid"{marking}, curve={height=18pt}, from=1-3, to=1-5]
			\arrow["{\prod_{i=1}^n \alpha_i}", shift right=3, shorten <=5pt, shorten >=5pt, Rightarrow, from=0, to=1]
		\end{tikzcd}\]
		where $i$ is the inclusion $\Dcat \subseteq \Ccat^n$ of the domain of $\theta$.
		
		Given a natural transformation $\alpha \colon \theta \Rightarrow \chi$ between partial operations $\theta, \chi \colon \Ccat^n \rightharpoondown \Ccat$ with the same domain $\Dcat \subseteq \Ccat^n$, i.e.\ an arrow $\theta \to \chi$ in $\End(\Ccat)_\partial(n)$, we define $\alpha_D \colon \chi_D \to \theta_D$ as the natural transformation whose component at $(F_1, \dots , F_n)$ is given by the transformation of profunctors
		\[\begin{tikzcd}[column sep=large]
			\Ccat && \Dcat & {\Ccat^n} & \1.
			\arrow[""{name=0, anchor=center, inner sep=0}, "{\Ccat(\chi,\id_\Ccat)}", "\shortmid"{marking}, curve={height=-18pt}, from=1-1, to=1-3]
			\arrow[""{name=1, anchor=center, inner sep=0}, "{\Ccat(\theta,\id_\Ccat)}"', "\shortmid"{marking}, curve={height=18pt}, from=1-1, to=1-3]
			\arrow["{\Ccat^n(\id_{\Ccat^n},i)}", "\shortmid"{marking}, from=1-3, to=1-4]
			\arrow["{F_1 \times \dots \times F_n}", "\shortmid"{marking}, from=1-4, to=1-5]
			\arrow["{\Ccat(\alpha, \id_\Ccat)}", shift right=3, shorten <=5pt, shorten >=5pt, Rightarrow, from=0, to=1]
		\end{tikzcd}\]
        The order of composition in this definition is not important as both choices are coherently isomorphic.
	\end{df}
	\begin{rem}\label{rem:partial_Day_extension_of_total_operation}
		Note that, by taking $i \colon \Dcat \hookrightarrow \Ccat^n$ as $\id_{\Ccat^n}$ in the above definition, we recover the description of the Day extension for total operations from \cref{df:day-extension}.
	\end{rem}
	\begin{rem}\label{rem:partial_day_as_coend}
		Let $\theta \colon \Ccat^n \rightharpoondown \Ccat$ be a partial operation and $F_1, \dots , F_n$ co-presheaves.  As a coend $\theta_D(F_1, \dots , F_n)$ is given by the formula
		\begin{align*}
		\theta_D(F_1, \dots , F_n)(a) & = \int^{\vec{d} \in \Dcat} \int^{\vec{c} \in \Ccat} \Ccat(\theta(\vec{d}),a) \times \Ccat^n(\vec{c},\vec{d}) \times \prod_{c_i \in \vec{c}} F_i(c_i), \\
		& \simeq \int^{\vec{c} \in \Ccat} \int^{\vec{d} \in \Dcat} \Ccat(\theta(\vec{d}),a) \times \Ccat^n(\vec{c},\vec{d}) \times \prod_{c_i \in \vec{c}} F_i(c_i).
		\end{align*}
	\end{rem}
	\begin{lem}\label{lem:Day_partial_func}
		For each $n$, the Day extension of partial operations yields a functor $(-)_D \colon \End(\Ccat)_\partial(n) \to \End([\Ccat \to \Set])(n)$.
	\end{lem}
	\begin{proof}
		This is a consequence of the functoriality of $\Cat\op \to \Prof$ on 2-cells, just as in \cref{lem:Day_ext_func_fixed_n}.
	\end{proof}
	\begin{lem}\label{lem:partial-multi-composite}
		Let $\theta \colon \Ccat^n \rightharpoondown \Ccat, \theta_i \colon \Ccat^{k_i} \rightharpoondown \Ccat$ be partial operations on a category $\Ccat$ (recall that, in particular, the domain of $\theta$ is a full subcategory of $\Dcat$).  Suppose that $(\theta_1 , \dots , \theta_n ) \colon \Ccat^{k_1 + \dots + k_n} \rightharpoondown \Ccat$ factors through $\Dcat \subseteq \Ccat$ (i.e.\ the multi-composite $\theta \circ (\theta_1, \dots , \theta_n)$ exists).  Then there is a canonical isomorphism $(\theta \circ (\theta_1, \dots , \theta_n))_D \cong \theta_D \circ ({\theta_1}_D , \dots , {\theta_n}_D)$.
	\end{lem}
	\begin{proof}
		Let $m = k_1 + \dots + k_n$.  By assumption, we have the following commuting diagram of functors as displayed on the left, from which we obtain the diagram of profunctors on the right.
		\[
		\begin{tikzcd}
			{\Ccat^n} & \Dcat & \Ccat \\
			{\Dcat_1 \times \dots \times \Dcat_n} \\
			{\Ccat^{m}},
			\arrow["i"', from=1-2, to=1-1]
			\arrow["\theta", from=1-2, to=1-3]
			\arrow["{(\theta_1, \dots , \theta_n)}", from=2-1, to=1-1]
			\arrow["{(\theta_1, \dots , \theta_n)'}"', dashed, from=2-1, to=1-2]
			\arrow["{(i_1, \dots , i_n)}"', hook, from=2-1, to=3-1]
		\end{tikzcd}
		\begin{tikzcd}
			{\Ccat^n} & \Dcat & \Ccat \\
			{\Dcat_1 \times \dots \times \Dcat_n} \\
			{\Ccat^{m}}.
			\arrow["{\Ccat^n((\theta_1 , \dots, \theta_n),1_{\Ccat^n}) \, }"', "\shortmid"{marking}, from=1-1, to=2-1]
			\arrow["{\Ccat^n(1_{\Ccat^n},i)}"', "\shortmid"{marking}, from=1-2, to=1-1]
			\arrow["{\Dcat((\theta_1, \dots , \theta_n)',1_\Dcat)}", "\shortmid"{marking}, from=1-2, to=2-1]
			\arrow["{\Ccat(\theta,1_\Ccat)}"', "\shortmid"{marking}, from=1-3, to=1-2]
			\arrow["{\Ccat^{m}(1_{\Ccat^m},(i_1, \dots , i_n)) \, } "', "\shortmid"{marking}, from=2-1, to=3-1]
		\end{tikzcd}
		\]
		We observe that the diagram of profunctors on the right also commutes (up to coherent isomorphism).  This follows from the bifunctoriality of the embedding of $\Cat\op$ in $\Prof$, and since 
        $\Ccat^n(i,1_{\Ccat^n})\circ\Ccat^n(1_{\Ccat^n},i) = \Ccat^n (i,i) = \Dcat(1_\Dcat,1_\Dcat)$, with the last equality holding because  $\Dcat$ is a full subcategory of $\Ccat^n$.
		
		Given co-presheaves $F_1^1, \dots , F_n^{k_n} \colon \Ccat \slashedrightarrow \1$, for each $j$, ${\theta_j}_D(F_j^1, \dots , F_j^{k_j})$ is given (up to coherent isomorphism) by the composite profunctor $(F_j^1 \times \dots \times F_j^{k_j}) \circ \Ccat^n(1_{\Ccat^n},i_j) \circ \Ccat(\theta_j,1_\Ccat) \colon \Ccat \slashedrightarrow \1$, and hence $\theta_D \circ ({\theta_1}_D, \dots , {\theta_n}_D)(F_1^1 , \dots , F_n^{k_n})$ is (isomorphic to) the composite profunctor
		\[
		(F_1^1 \times \dots \times F_n^{k_n}) \circ \Ccat^{m}(1_{\Ccat^m}, (i_1, \dots , i_n)) \circ \Ccat^n((\theta_1, \dots , \theta_n), 1_{\Ccat^n}) \circ \Ccat^n(1_{\Ccat^n},i) \circ \Ccat(\theta,1_\Ccat) \colon \Ccat \slashedrightarrow \1.
		\]
        Note that this composite is itself defined only up to coherent isomorphism since we have left the order of the composition unspecified. 
        
		Meanwhile, the partial composite $\theta \circ (\theta_1, \dots , \theta_n)$ is the span
		\[
		\begin{tikzcd}[column sep= large]
			\Ccat^{m} & \ar[hook']{l}[']{(i_1, \dots , i_n)} \Dcat_1 \times \dots \times \Dcat_n \ar{r}{\theta \circ (\theta_1, \dots , \theta_n)'} & \Ccat
		\end{tikzcd}
		\]
		and so $(\theta \circ (\theta_1, \dots , \theta_n))_D(F_1^1, \dots , F_n^{k_n})$ is given by the composite profunctor
		\[
		(F_1^1 \times \dots \times F_n^{k_n}) \circ \Ccat^{m}(1_{\Ccat^m}, (i_1, \dots , i_n)) \circ \Ccat(\theta \circ (\theta_1, \dots , \theta_n)', 1_{\Ccat}) \colon \Ccat^{m} \slashedrightarrow \1.
		\]
		By the bifunctoriality of $\Cat\op \to \Prof$ and the isomorphism $\Dcat((\theta_1 , \dots , \theta)',1_\Dcat) \cong \Ccat^n((\theta_1, \dots , \theta_n), 1_{\Ccat^n}) \circ \Ccat^n(1_{\Ccat^n},i)$ established above, there are canonical isomorphisms
		\begin{align*}
			& (F_1^1 \times \dots \times F_n^{k_n}) \circ \Ccat^{m}(1_{\Ccat^m}, (i_1, \dots , i_n)) \circ \Ccat(\theta \circ (\theta_1, \dots , \theta_n)', 1_\Ccat) \\
			\cong &  (F_1^1 \times \dots \times F_n^{k_n}) \circ \Ccat^{m}(1_{\Ccat^m}, (i_1, \dots , i_n)) \circ \Dcat((\theta_1 , \dots , \theta)',1_\Dcat) \circ \Ccat(\theta,1_\Ccat), \\
			 \cong & (F_1^1 \times \dots \times F_n^{k_n}) \circ \Ccat^{m}(1_{\Ccat^m}, (i_1, \dots , i_n)) \circ \Ccat^n((\theta_1, \dots , \theta_n), 1_{\Ccat^n}) \circ \Ccat^n(1_{\Ccat^n},i) \circ \Ccat(\theta,1_{\Ccat}).
		\end{align*}
		The latter is isomorphic to $\theta_D \circ ({\theta_1}_D, \dots , {\theta_n}_D)(F_1^1 , \dots , F_n^{k_n})$.  This yields the desired natural isomorphism $(\theta \circ (\theta_1, \dots , \theta_n))_D \cong \theta_D \circ ({\theta_1}_D , \dots , {\theta_n}_D)$.
	\end{proof}
	\begin{thm}\label{thm:day-ext-partial}
		Day extension of partial operations defines a pseudo-morphism of partial $\Cat$-operads
		\[
		(-)_D \colon \End(\Ccat)_\partial \to \End([\Ccat \to \Set])\op.
		\]
		Moreover, the triangle of pseudo-morphisms
		\begin{equation}\label{eq:partial-extensions-extends-extension}
		\begin{tikzcd}
			\End(\Ccat) \ar[hook]{d} \ar{rd}{(-)_D} & \\
			\End(\Ccat)_\partial \ar{r}{(-)_D} & \End([\Ccat \to \Set])\op 
		\end{tikzcd}
		\end{equation}
		commutes.
	\end{thm}
	\begin{proof}
		By \cref{lem:Day_partial_func}, for each $n$, $(-)_D \colon \End(\Ccat)_\partial(n) \to \End([\Ccat \to \Set])\op(n)$ is a functor.  Just as in \cref{thm:day_ext_is_operadic}, the unit is preserved up to (canonical) isomorphism since $\Ccat(1_\Ccat,1_\Ccat) \colon \Ccat \slashedrightarrow \Ccat$ is the identity profunctor on $\Ccat$, while the preservation up to (canonical) isomorphism for those multi-composites that exist is given in \cref{lem:partial-multi-composite}.  Finally, the commutativity of \cref{eq:partial-extensions-extends-extension} follows from \cref{rem:partial_Day_extension_of_total_operation}.
	\end{proof}

    \begin{coro}\label{coro:Day-algebras-partial}
       As in Corollary \ref{coro:Day-algebras}, let $\Sigma$ be an algebraic signature (i.e.\ a collection of operations with finite arity) and $E$ a collection of equations $E$ over $\Sigma$. Given a pseudo-partial-algebra $\Ccat$ of $(\Sigma,E)$, where the equations in $E$ hold up to strong equality, then the Day extension of each operation $\theta \colon \Ccat^n \to \Ccat$ in $\Sigma$ defines the structure of a (total) pseudo-algebra for $(\Sigma,E)$ on $[\Ccat \to \Set]$. 
    \end{coro}
    
	Thus, Day extension also permits the extension of partial algebraic structure, and in doing so it takes a partial operator into a total one. 
    An example of this, that we discuss further in \cref{ex:sep_logic}, is the separating conjunction ($\ast$) of separation logic, which is the Day extension of the partial binary operation $\sqcup \colon \H^2 \to \H$ that takes the union of \emph{disjoint} heaps.  The resultant operation on propositions $\ast$ is total despite the partiality of $\sqcup$. 
    
    When discussing partial operations, there are two notions of equality that are commonly considered.
	\begin{df}\label{df:stongequality} Let $\theta , \psi \colon \Ccat^n \rightharpoondown \Ccat$ be partial operations,
		\begin{enumerate}
			\item then $\theta$ and $\psi$ are said to be \emph{weakly equal} if they agree on the intersection of their domains,
			\item and $\theta$ and $\psi$ are said to be \emph{strongly equal} if they are weakly equal and have the same domain, i.e.\ they are identical functors.
		\end{enumerate}
	\end{df}
	\begin{df}
		By a \emph{partial algebra} for a operad $P$, we mean a category $\Ccat$ and a strict morphism of partial $\Cat$-operads $P \to \End(\Ccat)_\partial$.  Similarly, we will talk of \emph{partial pseudo-algebras} and \emph{partial lax-algebras}.
	\end{df}
	In particular, if $\Sigma$ is a collection of operation symbols with arities, and $E$ is a collection of equations over $\Sigma$, then a partial algebra $P_{(\Sigma,E)} \to \End(\Ccat)_\partial$ consists of a category $\Ccat$ and partial operations $\theta \colon \Ccat^n \rightharpoondown \Ccat$, for each operation symbol in $\Sigma$, for which the equations in $E$ are interpreted as strong equalities.  We note that the Day extension of partial operations preserves strong equality, up to isomorphism.  Thus, by post-composing a strict/pseudo/lax partial algebra structure $P \to \End(\Ccat)_\partial$ with the Day extension $(-)_D 
	\colon \End(\Ccat)_\partial \to \End([\Ccat \to \Set])\op$, we obtain a pseudo/pseudo/oplax partial algebra structure on $[\Ccat \to \Set]$.  
	\begin{rem}
		Weak equality between partial operations cannot be preserved by Day extension in general since the result is always a \emph{total} operation.  For example, the empty operation $0 \colon \emptyset \to \Ccat$ is weakly equal to any other $n$-ary operation on $\Ccat$, but $0_D$ is total and so cannot be weakly equal (up to isomorphism) to every $n$-ary operation on $[\Ccat \to \Set]$.
	\end{rem}
	\begin{ex}[Hybrid logic]\label{ex:hybrid_logic}
    This example and the next concern {Kripke semantics} for modal logic.  Recall that a \emph{Kripke frame} consists of a partially ordered set $\W$ that is normally interpreted as the `possible worlds'.  A propositional formula $\phi$ can either be true or false at a world $w \in \W$ (we write $w \Vdash \phi$ for the former) with the added monotonicity condition that if $w \leqslant v$ and $w \Vdash \phi$, then $v \Vdash \phi$.  Thus, propositional formulae can be equated with the objects of $[\W \to \2]$, i.e.\ the \emph{subterminal} objects of $[\W \to \Set]$.  The poset $[\W \to \2]$ inherits the structure of a Heyting algebra from the cartesian closed structure on $[\W \to \Set] = \Prof(\W,\1)$.  
		
		Hybrid logic extends modal logic via the addition of \emph{nominals} that allow reference to specific elements of the Kripke frame $\W$ (for further details, see \cite[\S 1, \S 8]{hybrid}).  Let $a$ be an element of $\W$ and $\phi$ a proposition.  The \emph{nominal} $\a$ is the proposition in $[\W \to \2]$ where $w \Vdash \a$ if and only if $ a \leqslant w$.  We also add the unary operation $@_a$, where $w \Vdash @_a \phi$ if and only if $a \Vdash \phi$, i.e.\ the proposition $@_a \phi$ expresses ``$\phi$ is true in world $a$''.
		
		We recognise that $\a \colon \1 \to [\W \to \2]$ and $@_a \colon [\W \to \2] \to [\W \to \2]$ are examples of Day extensions.
		\begin{enumerate}
			\item Let $\iota_a \colon \1 \to \W$ be the inclusion of the element $a \in \W$.    The profunctor $\W(\iota_a,1) \colon \W \slashedrightarrow \1$, as a function $\W \to \2 \subseteq \Set$, sends $w \in \W$ to $\top \in \2$ if and only if $a \leqslant w$.  Thus, $\W(\iota_a,1)$, which is the Day extension $({\iota_a})_D$, yields the interpretation of the nominal $\a$.
			\item Assuming that $\W$ is \emph{rooted} (\cite[\S 2.3]{modallogic}), i.e.\ $\W$ has a bottom element $w_0 \in \W$, consider the partial function $\iota_{w_0/a} \colon \W \rightharpoondown \W$ whose domain is $\iota_{a} \colon \1 \to \W$ and whose action is $\iota_{w_0} \colon \1 \to \W$.  It is readily calculated that the composite profunctor
			\[\begin{tikzcd}
				\W & \1 & \W & \1
				\arrow["{\W(\iota_{w_0},1)}", "\shortmid"{marking}, from=1-1, to=1-2]
				\arrow["{\W(1,\iota_{a})}", "\shortmid"{marking}, from=1-2, to=1-3]
				\arrow["\phi", "\shortmid"{marking}, from=1-3, to=1-4]
			\end{tikzcd}\]
			is precisely $@_a \phi$, and thus the Day extension $(\iota_{w_0/a})_D$ acts on propositions as $@_a$.
		\end{enumerate}
	\end{ex}
    \begin{ex}[Linear temporal logic]
        Linear temporal logic (LTL) has long been employed in program verification \cite{pnueli}.  In LTL, our set of worlds $W$, interpreted as time stamps, is discretely ordered and equipped with a \emph{successor} function $s \colon W \to W$ that represents a discrete time step.  We recognise that the Day extension $s_D \colon [W \to \2] \to [W \to \2]$ acts on a proposition $\phi$ as the `next' operator, usually written as ${\bf X} \phi$ \cite[\S 3.2]{huth_ryan}; that is, $w \Vdash {\bf X} \phi$ if and only if $s w \Vdash \phi$, or in other words ${\bf X} \phi$ expresses ``$\phi$ is satisfied in the next time step''.

        Using \cref{thm:day_ext_is_operadic} and \cref{thm:day-ext-partial}, we can generate the logical congruences obtained when mixing the operators of hybrid logic and LTL.  As a simple example, given $a \in W$, we have that $s_D \circ (\iota_a)_D \cong (s \circ \iota_a)_D = (\iota_{s a})_D $, expressing the tautology that ${\bf X}\a$ is logically equivalent to the nominal on $s a$.
    \end{ex}
	\begin{ex}[Separation logic]\label{ex:sep_logic}
		Separation logic was introduced for reasoning about programs which alter data structures.  The key idea is adding a ``separating conjunction'' $\phi \ast \psi$ between propositions that asserts that $\phi$ and $\psi$ hold on separate parts of the data structure \cite{separation,reynolds}.  The standard semantics for separation logic consists of a set $\H$ called \emph{heaps}, equipped with a partial binary operation $\sqcup \colon \H^2 \rightharpoondown \H$, whose domain $\Dcat$ we call the \emph{disjoint} heaps; $\sqcup$ is intended to be the union of disjoint heaps.  We write $h \# h'$ if the pair $h,h'$ are in the domain $\Dcat \subseteq \H^2$ of $\sqcup$.
		
		Given propositions $\phi, \psi \in [\H \to \2]$ on heaps, their \emph{separating conjunction} $\phi \ast \psi \in [\H \to \2]$ is the proposition where $h \Vdash \phi \ast \psi$ if and only if there exist $h_1$ and $h_2$ such that $h = h_1 \sqcup h_2$, $h_1 \Vdash \phi$ and $h_2 \Vdash \psi$, i.e.\ the heap $h$ can be decomposed into disjoint parts $h_1$ and $h_2$ on which $\phi$ and $\psi$ hold respectively.  We readily calculate that $\phi \ast \psi$, as a profunctor, is given by the composite
		\[\begin{tikzcd}
			\H & \Dcat & {{\H}^2} & \1.
			\arrow["{\H(\sqcup,1_\H)}", "\shortmid"{marking}, from=1-1, to=1-2]
			\arrow["{\H^2(i,1_{\H^2})}", "\shortmid"{marking}, from=1-2, to=1-3]
			\arrow["{\phi \times \psi}", "\shortmid"{marking}, from=1-3, to=1-4]
		\end{tikzcd}\]
		Thus, the separating conjunction $\ast$ of propositions is precisely given by the Day extension $\sqcup_D \colon [\H \to \2]^2 \to [\H \to \2]$.
		
		If we instead chose $\H$ to be a partially ordered set, and took $\sqcup \colon \H^2 \rightharpoondown \H$ as a monotone partial binary function, by the same construction we would recover the intuitionistic variant of separation logic (see, for instance, \cite[\S 9]{ishtiaq-ohearn}).
	\end{ex}
	\subsection{Adjoints to Day extensions}
		Recall that, given a monoidal category $(\Ccat, \otimes , I)$, the Day monoidal product $\otimes_D$ on $[\Ccat \to \Set]$ is \emph{residuated}, i.e.\ there are binary operations, usually denoted by $- / -, -\backslash - \colon [\Ccat \to \Set]^2 \to [\Ccat \to \Set]$, such that there are adjunctions $F \otimes_D - \dashv F \backslash -$ and $- \otimes_D F \dashv - / F$ for all $F \in [\Ccat \to \Set]$ (in \cite{day} and \cite[\S VII.7]{CWM}, this is called a \emph{(bi-)closed} monoidal structure, whereas we are following the terminology from classical order theory, for instance \cite{residuated}). Thus, Day convolution can be used to generate a class of models for \emph{bunched implication logic} \cite[\S 3.1]{bunched˙implications˙99} (by also using the cartesian closed monoidal structure on $[\Ccat \to \Set]$). 
		
		We now observe that, just as there was nothing special about the monoidal product when defining its Day extension, the existence of adjoints to Day convolution also extends to all algebraic and partial algebraic structure.
	\begin{prop}\label{prop:day_ext_closed}
		Let $\theta \colon \Ccat^n \rightharpoondown \Ccat$ be a partial operation.  The Day extension $\theta_D $ is \emph{residuated} in each variable.  That is, for each $j$ there is an operation 
        \[R^j_\theta \colon [\Ccat \to \Set]^n \to [\Ccat \to \Set]\]
        such that, for all tuples of co-presheaves $F_1, \dots , F_n , G \colon \Ccat \to \Set$, there is a natural isomorphism
		\[
		[\Ccat \to \Set]\left(\theta_D(F_1, \dots , F_n),G\right)  \cong [\Ccat \to \Set]\left(F_j,R^j_\theta(G, F_1, \dots , F_{j-1}, F_{j+1} , \dots , F_n)\right).
		\]
	\end{prop}
	\begin{proof}
		Our proof mimics the argument given in, for instance, \cite[Remark 6.2.4]{fosco}.  We take $R^j_\theta(G, F_1, \dots , F_{j-1}, F_{j+1} , \dots , F_n) \colon \Ccat \to \Set$ to be the co-presheaf where $R^i_\theta(G, F_1, \dots , F_{j-1}, F_{j+1} , \dots , F_n)(a)$ is given by the end:
		\[
		 \int_{\vec{c}\in \Ccat^n, \vec{d} \in \Dcat} \left[\left( \Ccat^n(\vec{c}\,[a/c_j],\vec{d}) \times \Ccat(\theta(\vec{d}),c_j) \times \prod_{i \neq j} F_i(c_i) \right) \to G(c_j)\right],
		\]
		where $\vec{c} \, [a/c_j]$ represents the tuple $(c_1, \dots , c_{j-1},a,c_{j+1}, \dots , c_n)$ and $[A \to B]$ represents the set of functions from a set $A$ to a set $B$ (i.e.\ the internal hom in $\Set$).
		
		The desired isomorphism is then obtained via `(co-)end yoga':
		\begin{align*}
			& [\Ccat \to \Set]\left(\theta_D(F_1, \dots , F_n),G\right) \\
		 \cong & \int_{a \in \Ccat} [\theta_D(F_1, \dots , F_n)(a) \to G(a)], \\
			 \cong & \int_{a \in \Ccat} \left[\left(\int^{\vec{d} \in \Dcat, \vec{c}\in \Ccat^n} \Ccat(\theta(\vec{d}),a) \times  \Ccat^n(\vec{c},\vec{d}) \times \prod_{c_i \in \vec{c}} F_i(c_i) \right) \to G(a)\right], \\
			 \cong & \int_{a \in \Ccat} \int_{\vec{d} \in \Dcat, \vec{c} \in \Ccat^n} \left[\left( \Ccat(\theta(\vec{d}),a) \times  \Ccat^n(\vec{c},\vec{d}) \times \prod_{c_i \in \vec{c}} F_i(c_i) \right) \to G(a)\right], \\
			 \cong & \int_{\vec{d}\in \Dcat, a, c_1 , \dots , c_n\in \Ccat}\left[F_j(c_j) \to \left[\left(\Ccat(\theta(\vec{d}),a) \times \Ccat^n(\vec{c},\vec{d}) \times \prod_{{{c_i \in \vec{c}}, { i \neq j}}} F_i(c_i)\right) \to G(a)\right]\right], \\
			 \cong & \int_{c_j \in \Ccat} \left[F_j(c_j) \to \int_{\substack{{\vec{d} \in \Dcat,} \\ { \vec{c}\,[a/c_j] \in \Ccat^n}}} \left[\left(\Ccat(\theta(\vec{d}),a) \times \Ccat^n(\vec{c},\vec{d}) \times \prod_{{{c_i \in \vec{c}}, { i \neq j}}} F_i(c_i)\right) \to G(a)\right] \right], \\
			 \cong & \int_{c_j \in \Ccat} \left[F_j(c_j) \to R^j_\theta(G, F_1 , \dots, F_{i-1}, F_{i+1}, \dots , F_n)(c_j)\right], \\
			 \cong & [\Ccat \to \Set] (F_j, R^j_\theta(G, F_1 , \dots, F_{j-1}, F_{j+1}, \dots , F_n)),
		\end{align*}
		where we have used \cite[Corollary 1.2.8 \& Theorem 1.3.1]{fosco} and \cref{rem:partial_day_as_coend}.
	\end{proof}
	\begin{ex}[Separation logic, continued]\label{ex:sep_logic_cont}
	In addition to the separating conjunction discussed in \cref{ex:sep_logic}, separation logic also has a \emph{separating implication} $\phi\sepimp \psi$ , where $-\ast \phi \dashv \phi \sepimp -$ (see \cite[\S 2]{separation}).  The right adjoint $\phi \sepimp -$ to $- \ast \phi$ can be obtained by an application of \cref{prop:day_ext_closed},  which yields that
	\begin{align*}
	(\phi \sepimp \psi)(h) & = \int_{\substack{{h_1, h_2 \in \H} \\ {h'_1,h'_2 \in \Dcat}}} \left[\left(\H^2((h_1,h),(h'_1,h'_2)) \times \H(h'_1 \sqcup h'_2, h_2) \times \phi(h_1)\right) \to \psi(h_2)\right], \\
	& \cong \int_{h' \in \H, h' \# h} [\phi(h') \to \psi(h' \sqcup h)].
	\end{align*}
		That is to say, we recover the standard interpretation of the separating implication where $h \Vdash \phi\sepimp \psi$ if and only if, for every $h'$ with $h' \# h$, if $h' \Vdash \phi$ then $h' \sqcup h \Vdash \psi$.  
	\end{ex}
	\subsection{Replacing partial operations by spans -- concluding remarks}\label{remarks_on_spans}
				A reader may wonder why we have chosen not to define the composite of partial functors as the outside span in the diagram
				\[
				\begin{tikzcd}
						\Ecat \times_\Dcat \Ecat' \ar[hook]{d} \ar{r}{G'} & \Ecat \ar{r}{F} \ar[hook]{d} & \Ccat \\
						\Ecat' \ar{r}[']{G} \ar[hook]{d}  & \Dcat & \\
						\Dcat',
						\arrow["\lrcorner"{anchor=center, pos=0.125}, draw=none, from=1-1, to=2-2]
					\end{tikzcd}
				\]
				where $G'$ is the (strict) pullback of $G$ along $\Ecat \hookrightarrow \Dcat$, i.e.\ $\Ecat \times_\Dcat \Ecat' \subseteq \Ecat'$ is the subcategory consisting of those objects and arrows whose image under $G$ lies in $\Ecat \subseteq \Dcat$.  (Note that if $G \colon \Ecat' \to \Dcat$ factors through $\Ecat$, as in \cref{eq:composition_of_partial_functors}, then the two notions of composite for partial functors coincide.)  
				
				The issue lies in the fact that (strict) pullback squares of functors are not \emph{exact} in the sense of \cite{guitart-exact}, meaning that the canonical natural transformation, or \emph{mate}, between the corresponding profunctors
					\[\begin{tikzcd}
						{\Ecat \times_\Dcat \Ecat'} & \Ecat \\
						{\Ecat'} & \Dcat
						\arrow["{\Ecat'(1,F')}"', "\shortmid"{marking}, from=1-1, to=2-1]
						\arrow["{\Ecat(G',1)}"', "\shortmid"{marking}, from=1-2, to=1-1]
						\arrow["{\Dcat(1,F)}", "\shortmid"{marking}, from=1-2, to=2-2]
						\arrow["{\Dcat(G,1)}", "\shortmid"{marking}, from=2-2, to=2-1]
						\arrow["\varepsilon", Rightarrow, from=1-1, to=2-2]
					\end{tikzcd}\]
					is not a natural isomorphism in general.  This means that, if we were to perform an identical argument to \cref{lem:partial-multi-composite} and \cref{thm:day-ext-partial}, we would only obtain a \emph{lax} morphism of operads.  More explicitly:
					\begin{enumerate}
						\item Given partial operations $\theta \colon \Ccat^n \rightharpoondown \Ccat, \theta_i \colon \Ccat^{k_i} \rightharpoondown \Ccat$ (whose domains we no longer require to be full subcategories), we define a new multi-composite $\circ^\pb$ where $\theta \circ^\pb (\theta_1 , \dots , \theta_n)$ is given by the outside span in the diagram
						\[\begin{tikzcd}[column sep=large]
							\Ecat & \Dcat & \Ccat \\
							{\Dcat_1 \times \dots \times \Dcat_n} & {\Ccat^n} \\
							{\Ccat^{k_1 + \dots + k_n}}.
							\arrow["{(\theta_1, \dots ,\theta_n)|_\Ecat}", from=1-1, to=1-2]
							\arrow["{i'}"', hook, from=1-1, to=2-1]
							\arrow["\lrcorner"{anchor=center, pos=0.125}, draw=none, from=1-1, to=2-2]
							\arrow["\theta", from=1-2, to=1-3]
							\arrow["i", hook, from=1-2, to=2-2]
							\arrow["{(\theta_1, \dots , \theta_n)}"', from=2-1, to=2-2]
							\arrow["{(i_1, \dots, i_n)}"', hook, from=2-1, to=3-1]
						\end{tikzcd}\]
						With this definition of multi-composite, we obtain a $\Cat$-operad $\End(\Ccat)_\partial^\pb$ which differs from the partial $\Cat$-operad $\End(\Ccat)_\partial$ given in \cref{df:partial_endo_operad} only in the definition of the multi-composite.
						\item Let $F_1^1, \dots , F_n^{k_n}$ be co-presheaves, and consider the diagram
						\begin{equation*}\label{eq:lax-preservation-of-multicomp}
							\begin{tikzcd}[column sep=large 
                            ]
								\Ecat && \Dcat & \Ccat \\
								{\Dcat_1 \times \dots \times \Dcat_n} && {\Ccat^n} \\
								{\Ccat^{k_1 + \dots k_n}} \\
								\1,
								\arrow["{\Dcat_1 \times \dots \times \Dcat_n(\id , i')\, }"', "\shortmid"{marking}, from=1-1, to=2-1]
								\arrow["\varepsilon", 
								shorten <=10pt, shorten >=10pt, 
								Rightarrow, from=1-1, to=2-3]
								\arrow["{\Dcat((\theta_1, \dots, \theta_n)|_\Ecat, \id)}"', "\shortmid"{marking}, from=1-3, to=1-1]
								\arrow["{\Ccat^n(\id,i)}", "\shortmid"{marking}, from=1-3, to=2-3]
								\arrow["{\Ccat(\theta, \id)}"', "\shortmid"{marking}, from=1-4, to=1-3]
								\arrow["{\Ccat^{k_1 + \dots + k_n}(\id, (i_1, \dots , i_n))} \, "', "\shortmid"{marking}, from=2-1, to=3-1]
								\arrow["{\Ccat^n((\theta_1, \dots, \theta_n),\id)}", "\shortmid"{marking}, from=2-3, to=2-1]
								\arrow["{F_1^1 \times \dots \times F_n^{k_n}}\, "', "\shortmid"{marking}, from=3-1, to=4-1]
							\end{tikzcd}
						\end{equation*}
						where $\varepsilon$ is the induced mate.  The outside composite profunctor $\Ccat \slashedrightarrow \1$, along the top horizontal followed by the leftmost vertical, corresponds to $(\theta \circ^\pb (\theta_1, \dots , \theta_n))_D(F_1^1, \dots , F_n^{k_n})$, whereas the other composite is naturally isomorphic to $(\theta_D) \circ ({\theta_1}_D , \dots , {\theta_n}_D)(F_1^1 , \dots , F_n^{k_n})$. The mate $\varepsilon$ yields a natural transformation 
						\[\eta_{\theta, \theta_1, \dots , \theta_n} \colon (\theta_D) \circ ({\theta_1}_D , \dots , {\theta_n}_D) \to (\theta \circ^\pb (\theta_1, \dots , \theta_n))_D\]
						which forms the underlying data of only a lax-morphism of $\Cat$-operads $(-)_D \colon \End(\Ccat)_\partial^\pb \to \End([\Ccat \to \Set])\op$.  Thus, postcomposing by the Day extension $(-)_D \colon \End(\Ccat)_\partial^\pb \to \End([\Ccat \to \Set])\op$ a lax algebra structure on $\Ccat$ is sent to an oplax algebra structure on $[\Ccat \to \Set]$, but pseudo-algebras are not necessarily preserved.
					\end{enumerate}

					In contrast, \emph{comma squares} are exact.  However, this means abandoning partial operations in favour of arbitrary spans.  We briefly sketch how our constructions generalise for spans.  We define a $\Cat$-operad of spans $\Span(\Ccat)$ where $\Span(\Ccat)(n)$ is the category whose objects are spans between $\Ccat^n$ and $\Ccat$, and whose morphisms are natural transformations of the right leg of a span with common left leg, as in the diagrams
					\[
					\begin{tikzcd}
						\Ccat^n & \ar{l}[']{F} \Dcat \ar{r}{G} & \Ccat
					\end{tikzcd}, \quad 
					\begin{tikzcd}
						\Ccat^n & \Dcat & \Ccat.
						\arrow["F"', from=1-2, to=1-1]
						\arrow[""{name=0, anchor=center, inner sep=0}, "G", curve={height=-18pt}, from=1-2, to=1-3]
						\arrow[""{name=1, anchor=center, inner sep=0}, "{G'}"', curve={height=18pt}, from=1-2, to=1-3]
						\arrow["\alpha", shorten <=3pt, shorten >=3pt, Rightarrow, from=0, to=1]
					\end{tikzcd}
					\]
					The unit is the identity span on $\Ccat$, while the multi-composite of spans $(F,G) \in \Span(\Ccat)(n), (F_1,G_1) \in \Span(\Ccat)(k_1) , \dots , (F_n,G_n) \in \Span(\Ccat)(n)$ is given by the outside span in the diagram
					\[
					\begin{tikzcd}[column sep=large]
						{F \downarrow (G_1, \dots , G_n)} & \Dcat & \Ccat \\
						{\Dcat_1 \times \dots \times \Dcat_n} & {\Ccat^n} \\
						{\Ccat^{k_1 + \dots + k_n}},
						\arrow["{(G'_1, \dots , G'_n)}", from=1-1, to=1-2]
						\arrow["{F'}"', from=1-1, to=2-1]
						\arrow["G", from=1-2, to=1-3]
						\arrow["F", from=1-2, to=2-2]
						\arrow["{(G_1, \dots , G_n)}"', from=2-1, to=2-2]
						\arrow["{(F_1, \dots, F_n)}"', from=2-1, to=3-1]
					\end{tikzcd}
					\]					
					where ${F \downarrow (G_1, \dots , G_n)}$ is the comma object (see \cite[\S II.6]{CWM}).  Emulating \cref{df:partial-Day-extension}, we define the \emph{Day extension of a span} $\Ccat^n \xleftarrow{F} \Dcat \xrightarrow{G} \Ccat$ as the functor $[\Ccat \to \Set]^n \to [\Ccat \to \Set]$ that sends co-presheaves $F_1, \dots, F_n \in [\Ccat \to \Set]$ to the profunctor
					\[
					\begin{tikzcd}[column sep = large]
						\Ccat \ar["\shortmid"{marking}]{r}{\Ccat(G,1_\Ccat)} & \Dcat \ar["\shortmid"{marking}]{r}{\Ccat^n(1_{\Ccat^n},F)}  & \Ccat^n \ar["\shortmid"{marking}]{r}{F_1 \times \dots \times F_n} & \1,
					\end{tikzcd}
					\]
					and a tuple of natural transformations $\alpha_i \colon F_i \Rightarrow G_i$ between co-presheaves to the transformation of profunctors
					\[\begin{tikzcd}[column sep=large]
						\Ccat & \Dcat \ar["\shortmid"{marking}]{r}{\Ccat^n(1_{\Ccat^n},F)} & {\Ccat^n} && {\1},
						\arrow["{\Ccat(G,1_\Ccat)}", "\shortmid"{marking}, from=1-1, to=1-2]
						\arrow[""{name=0, anchor=center, inner sep=0}, "{F_1 \times \dots \times F_n}", "\shortmid"{marking}, curve={height=-18pt}, from=1-3, to=1-5]
						\arrow[""{name=1, anchor=center, inner sep=0}, "{G_1 \times \dots \times G_n}"', "\shortmid"{marking}, curve={height=18pt}, from=1-3, to=1-5]
						\arrow["{\prod_{i=1}^n \alpha_i}", shift right=3, shorten <=5pt, shorten >=5pt, Rightarrow, from=0, to=1]
					\end{tikzcd}\]
					Following an identical argument to the above, and using that comma squares are exact, we obtain a pseudo-morphism of $\Cat$-operads 
					\[(-)_D \colon \Span(\Ccat) \to \End([\Ccat \to \Set])\op.\]
			\begin{ex}
				As a final example, we revisit the $@_a$ operation from hybrid logic discussed in \cref{ex:hybrid_logic}, and obtain it as the Day extension of a span in the unrooted case.  Let $\W$ be a (potentially unrooted) Kripke frame, and consider the span
				\[
				\begin{tikzcd}
					\W & \ar{l}[']{c_a} \W \ar[equal]{r} & \W
				\end{tikzcd}
				\]
				where $c_a \colon \W \to \W$ is the function that sends all of $\W$ to a chosen element $a \in \W$.  The Day extension of this span, applied to a proposition $\phi \in [\W \to \2]$, i.e.\ the composite profunctor $\phi \circ \W(1,c_a) \colon \W \slashedrightarrow \1$, is readily calculated to coincide with the proposition $@_a \phi$ described in \cref{ex:hybrid_logic}.
			\end{ex}

	\appendix
	\crefalias{section}{appendix}
	\section{Coherence axioms for morphisms of operads}\label{appendix:coherence}
	We return to fully stating the coherence axioms for a lax-morphism of $\Cat$-operads from \cref{df:pseudomorph_of_cat-operads}.  These are modelled on the standard coherence axioms for, say, pseudo-functors \cite[Definition 4.1]{benaboubicats}.  Recall that a lax-morphism $\omega \colon P \to Q$ consists of:
	\begin{enumerate}
		\item a functor $\omega_n \colon P(n) \to Q(n)$ for each $n \in \N$,
		\item a morphism $\lambda \colon 1 \to \omega_1(1) \in Q(1)$,
		\item and, for each $\theta \in P(n), \theta_1 \in P(k_1), \dots , \theta_n \in P(k_n)$, a morphism
		\[
		\eta_{\theta, \theta_1, \dots , \theta_n} \colon \omega_n(\theta) \circ (\omega_{k_1}(\theta_1), \dots , \omega_{k_n}(\theta_n)) \to \omega_{k_1 + \dots + k_n} (\theta \circ (\theta_1 , \dots , \theta_n)), 
		\]
		natural in $\theta, \theta_i$,
	\end{enumerate}
	These must satisfy the following coherence axioms.  Firstly, the composite
	\[\begin{tikzcd}[row sep=large]
		{\omega_n(\theta) \circ (\omega_{k_1}(\theta_1) \circ (\omega_{m_1^1}(\theta_1^1), \dots , \omega_{m_1^{k_1}}(\theta_1^{k_1})), \dots , \omega_{k_n}(\theta_n) \circ (\omega_{m_n^1}(\theta_n^1), \dots , \omega_{m_n^{k_n}}(\theta_n^{k_n}))} \\
		{\omega_n(\theta) \circ (\omega_{m_1^1 + \dots + m_1^{k_1}}(\theta_1 \circ (\theta_1^1 , \dots, \theta_1^{k_1})), \dots , \omega_{m_n^1 + \dots + m_n^{k_n}}(\theta_n \circ (\theta_n^1 , \dots , \theta_n^{k_n})))} \\
		{\omega_{m_1^1 + \dots + m_n^{k_n}}(\theta \circ(\theta_1 \circ (\theta_1^1, \dots, \theta_1^{k_1}), \dots , \theta_n\circ(\theta_n^1, \dots, \theta_n^{k_n})))}
		\arrow["{\circ\left(\id_{\omega_n(\theta)}, \eta_{\theta_1, \theta_1^1, \dots , \theta_1^{k_1}}, \dots, \eta_{\theta_n, \theta_n^1, \dots, \theta_n^{k_n}}\right)}"', from=1-1, to=2-1]
		\arrow["{\eta_{\theta, \theta_1 \circ (\theta_1^1, \dots, \theta_1^{k_1}), \dots , \theta_n \circ (\theta_n^1, \dots , \theta_n^{k_n})}}"', from=2-1, to=3-1]
	\end{tikzcd}\]
	is equal to the composite
	\[\begin{tikzcd}[row sep=large]
		{(\omega_n(\theta) \circ (\omega_{k_1}(\theta_1), \dots, \omega_{k_n}(\theta_n)) \circ (\omega_{m_1^1}(\theta_1^1), \dots , \omega_{m_n^{k_n}}(\theta_n^{k_n}))} \\
		{\omega_{k_1 + \dots + k_n} (\theta \circ (\theta_1 , \dots , \theta_n)) \circ (\omega_{m_1^1}(\theta_1^1), \dots, \omega_{m_n^{k_n}}(\theta_n^{k_n}))} \\
		{\omega_{m_1^1+ \dots m_n^{k_n}}((\theta \circ (\theta_1, \dots, \theta_n) \circ (\theta_1^1, \dots , \theta_n^{k_n}))}.
		\arrow["{\circ\left(\eta_{\theta, \theta_1, \dots , \theta_n},\id_{\omega_{m_1^1}(\theta_1^1)}, \dots, \id_{\omega_{m_n^{k_n}}(\theta_n^{k_n})}\right)}", from=1-1, to=2-1]
		\arrow["{\eta_{(\theta \circ (\theta_1, \dots, \theta_n)), \theta_1^1, \dots, \theta_n^{k_n}}}", from=2-1, to=3-1]
	\end{tikzcd}\]
	Note that the associativity axioms for $P$ and $Q$ ensure that the source and target of these composites agree.
	
	Secondly, the diagram
	\[\begin{tikzcd}[row sep = large]
		{\omega_n(\theta) \circ(1,\dots, 1)} & {\omega_n(\theta)} & {1 \circ \omega_n(\theta)} \\
		{\omega_n(\theta) \circ (\omega_1(1), \dots, \omega_1(1))} && {\omega_1(1) \circ \omega_n(\theta)} \\
		{\omega_n(\theta \circ (1, \dots, 1))} & {\omega_n(\theta)} & {\omega_n(1 \circ \theta)}
		\arrow[equals, from=1-1, to=1-2]
		\arrow["{\circ \left( \id_{\omega_n(\theta)}, \lambda, \dots ,\lambda \right)}"', from=1-1, to=2-1]
		\arrow[equals, from=1-2, to=1-3]
		\arrow[equals, from=1-2, to=3-2]
		\arrow["{\circ \left( \lambda , \id_{\omega_n(\theta)}\right)}", from=1-3, to=2-3]
		\arrow["{\eta_{\theta, 1,\dots, 1}}"', from=2-1, to=3-1]
		\arrow["{\eta_{1,\theta}}", from=2-3, to=3-3]
		\arrow[equals, from=3-1, to=3-2]
		\arrow[equals, from=3-2, to=3-3]
	\end{tikzcd}\]
	commutes.  The top and bottom horizontal equalities are provided by the unital axioms of $P$ and $Q$.

    \begin{prop}[cf.\ \cref{thm:day_ext_is_operadic}]\label{prop:coherence_of_ext}
        The Day extension from \cref{df:day-extension} satisfies the above coherence conditions.
    \end{prop}
    \begin{proof}
        We will content ourselves with demonstrating that the right-hand square in the above diagram commutes, i.e.\ that $\eta_{1,\theta} \circ \left( \lambda , \id_{\theta_D}\right)$ is the identity on $\theta_D$, as the other axioms follow by a similar manipulation of the coherence for bifunctors and bicategories. Recall that, in our context, the isomorphism $\lambda \colon 1_{[\Ccat \to \Set]} \to (1_\Ccat)_D = - \circ \Ccat(1_\Ccat,1_\Ccat) \in \End([\Ccat \to \Set])\op$ is given by the right unitor $\mathfrak{r}_{\Ccat,\1}$ of $\Prof$, while the isomorphism $\eta_{\theta, \theta_1, \dots , \theta_n} \colon \theta_D \circ ({\theta_1}_D, \dots {\theta_n}_D) \to ( \theta \circ (\theta_1, \dots , \theta_n))_D$ is given by the inverse to the composite $(- \ast \mathfrak{c}_{\theta, (\theta_1, \dots , \theta_n)} )
        \circ \mathfrak{a}_{\Ccat,\Ccat^n,\1} $, where $\mathfrak{a}_{\Ccat,\Ccat^n,\1}$ is the associator of $\Prof$ and the isomorphism 
        \[\mathfrak{c}_{\theta, (\theta_1, \dots , \theta_n)} \colon \Ccat^n((\theta_1, \dots , \theta_n), 1_{\Ccat^n}) \circ \Ccat(\theta,1_\Ccat) \to \Ccat(\theta \circ (\theta_1, \dots , \theta_n),1_\Ccat)\]
        is given by the bifunctoriality of $\Cat\op \to \Prof$.  We therefore desire to show that, for an operation $\theta \colon \Ccat^n \to \Ccat$, the following diagram of isomorphisms commutes in $\End([\Ccat \to \Set]) = \Prof(\Ccat,\1)$:
        \begin{equation}\label{eq:identity_coherence}
        \begin{tikzcd} 
            (1_\Ccat)_D \circ \theta_D (-) \ar[draw=none]{d}[marking, allow upside down]{=}
            \\[-10pt]
            (- \circ \Ccat(\theta,1_\Ccat))\circ \Ccat(1_\Ccat,1_\Ccat)  \ar{r}{\mathfrak{a}_{\Ccat,\Ccat,\1}}  \ar{rd}[']{\mathfrak{r}_{\Ccat,\1}} & - \circ (\Ccat(\theta,1_\Ccat) \circ \Ccat(1_\Ccat,1_\Ccat))  \ar{d}{- \ast \mathfrak{c}_{1_\Ccat,\theta}} \\
            &  - \circ \Ccat(1_\Ccat \circ \theta,1_\Ccat) &[-48pt] = 1_{[\Ccat \to \Set]} \circ \theta_D(-).
        \end{tikzcd}
        \end{equation}
        First, note that since $\Cat\op$ is a strict 2-category and the bifunctor $\Cat\op \to \Prof$ sends identities to identities on the nose, it follows from \cite[Definition 4.1(M.2)]{benaboubicats} that the isomorphism $\mathfrak{c}_{1_\Ccat,\theta} \colon \Ccat(\theta,1) \circ \Ccat(1,1) \to \Ccat(1 \circ \theta,1) = \Ccat(\theta,1)$ is given by the right unitor $\mathfrak{r}_{\Ccat,\1}$ as well.  The desired diagram \cref{eq:identity_coherence} now commutes as a consequence of the coherence conditions satisfied by the bicategory $\Prof$ (cf.\ \cite[Defintion 1.1]{benaboubicats} and \cite{kelly}).
    \end{proof}

	
		\renewcommand{\baselinestretch}{1}
	
\end{document}